\documentclass[11pt]{amsart}
\usepackage[normalem]{ulem}
\usepackage{booktabs}
\usepackage{makecell}
\usepackage[utf8]{inputenc}
\usepackage{fullpage, amsmath, amssymb, amsthm, enumerate, url, bibentry, hyperref, color,xcolor}
\usepackage[all,cmtip]{xy}
\usepackage{tikz,epsfig,pb-diagram}
\usepackage{subcaption}
\usepackage{verbatim}
\usepackage{mathtools}
\usepackage{colonequals}
\usepackage{tabularray}
\usepackage{tabularx}
\usepackage{array}

\usepackage{mathtools}
\mathtoolsset{showonlyrefs}
\usepackage{zref-clever}
\usepackage{longtable}
\newcolumntype{C}{>{\centering\arraybackslash}X}
\zcsetup{cap}
\usepackage[singlelinecheck=false]{caption}
\usepackage{tikz-cd}

\allowdisplaybreaks[4]            
\definecolor{purple}{rgb}{0.59, 0.44, 0.84}
\newcommand{\ft}[1]{{\color{cyan}  #1}}

\newcommand{\psa}[1]{{\color{brown}  #1}}

\theoremstyle{plain}
\newtheorem{theorem}{Theorem}[section]
\newtheorem*{theorem*}{Theorem}

\AddToHook{env/corollary/begin}{\zcsetup{countertype={theorem=corollary}}}
\newtheorem{corollary}[theorem]{Corollary}
\newtheorem*{corollary*}{Corollary}

\AddToHook{env/lemma/begin}{\zcsetup{countertype={theorem=lemma}}} 
\newtheorem{lemma}[theorem]{Lemma}
\newtheorem*{lemma*}{Lemma}

\newtheorem*{sublemma*}{}

\AddToHook{env/proposition/begin}{\zcsetup{countertype={theorem=proposition}}}
\newtheorem{proposition}[theorem]{Proposition}
\newtheorem*{proposition*}{Proposition}

\AddToHook{env/conjecture/begin}{\zcsetup{countertype={theorem=conjecture}}} 

\newtheorem*{conjecture*}{Conjecture}

\AddToHook{env/noname/begin}{\zcsetup{countertype={theorem=noname}}} 

\newtheorem*{noname*}{}

\theoremstyle{definition}
\AddToHook{env/definition/begin}{\zcsetup{countertype={theorem=definition}}} 
\newtheorem{definition}[theorem]{Definition}
\newtheorem*{definition*}{Definition}

\AddToHook{env/observation/begin}{\zcsetup{countertype={theorem=observation}}}

\newtheorem*{observation*}{Observation}

\AddToHook{env/example/begin}{\zcsetup{countertype={theorem=example}}} 
\newtheorem{example}[theorem]{Example}
\newtheorem*{example*}{Example}

\AddToHook{env/exercise/begin}{\zcsetup{countertype={theorem=exercise}}} 

\newtheorem*{exercise*}{Exercise}

\AddToHook{env/notation/begin}{\zcsetup{countertype={theorem=notation}}}

\newtheorem*{notation*}{Notation}

\AddToHook{env/remark/begin}{\zcsetup{countertype={theorem=remark}}}
\newtheorem{remark}[theorem]{Remark}
\newtheorem*{remark*}{Remark}

\numberwithin{equation}{section}

\def \l {\lambda}

\def\C{\mathbb C}
\def\G{\mathbf G}

\def\wp{\mathfrak{p}}

\def \Ind {\textnormal{Ind}}

\def\Tr{{\rm Tr}}
\def\Gal{{\rm Gal}}
\def\Frob{{\rm Frob}}
\def\F{{\mathbb F}}
\def\Z{{\mathbb Z}}

\def\O{\mathcal O}

\def\G{\Gamma}

\def\Im{\mathrm{Im}\,}
\def\Re{\mathrm{Re}\,}
\def\SL{\mathrm{SL}}

\def\N{\mathbb N}

\def\R{\mathbb{R}}
\def\Q{\mathbb{Q}}
\def\H{\mathbb{H}}
\def\O{\mathcal{O}}
\def\H{\mathbb H}

\def \ol{\overline}

\newcommand*\HYPERskip{&}
\catcode`,\active
\newcommand*\pFq{
\begingroup
\catcode`\,\active
\def ,{\HYPERskip}%
\doHyper
}
\catcode`\,12
\def\doHyper#1#2#3#4#5{%
\, _{#1}F_{#2}\left[\begin{matrix}#3 \smallskip \\  #4\end{matrix} \; ; \; #5\right]%
\endgroup
}

\def \pfq{\pFq}
\def\gen#1{\langle #1\rangle}

\catcode`,\active

\catcode`\,12

\catcode`,\active

\catcode`\,12

\catcode`,\active

\catcode`\,12

\def\Im{\mathrm{Im}\,}
\def\SL{\mathrm{SL}}

\def\N{\mathbb N}

\def \HD{\mathit{HD}}

\begin{document}

\title{Special $L$-values of certain CM weight three Hecke eigenforms}

\author{Paresh Arora, Koustav Mondal, Akio Nakagawa, Fang-Ting Tu}
\begin{abstract}
 
Ramanujan’s theory of elliptic functions to alternative bases  connects modular forms with hypergeometric series and has led to applications such as the modularity of certain hypergeometric Galois representations.  In this paper, we relate special values of $L$-functions of certain CM  Hecke eigenforms to Ramanujan’s alternative bases via the modularity of hypergeometric Galois representations associated with hypergeometric series $\pFq32{\frac12&\frac 1d&\frac{d-1}d}{&1&1}t$, $d=2$, $3$, $4$, and $6$,  arising from tensor products of CM elliptic curves over real quadratic fields. We also give a complete classification of these type of hypergeometric Galois representations.
 
\end{abstract}
\maketitle
\tableofcontents 

\section{Introduction} \label{sec:intro}

Special values of $L$-functions play a central role in modern number theory. They encode arithmetic information and serve as a bridge between analysis, algebra, and geometry, as illustrated by the Birch and Swinnerton–Dyer conjecture and its generalizations, such as the Bloch–Kato and Be\u{\i}linson–Bloch conjectures \cite[et al.]{Beilinson, Beilinson-Ktheory, Bloch, BlochKato, BSD1, BSD2}. These conjectures relate the analytic behavior of $L$-functions to arithmetic invariants, including ranks of elliptic curves and Selmer groups. Deligne’s conjecture further predicts precise algebraic properties of  special $L$-values attached to modular forms \cite{Deligne1979}. Such values are also closely connected to Bernoulli numbers, quadratic forms, and zeta functions of real quadratic fields, beginning with the work of Eichler and Shimura and further developed in the theory of Hilbert modular forms \cite[et al.]{choie2021, eichler,  kohnen_zagier_halfintegral, kohnen, Shimura-introduction}.

Let $f$ be a cusp form of weight $k$  with Fourier expansion $f=\displaystyle \sum_{n\geq 1} a_nq^n$. Its $L$-series 
$$
   L(f,s)\coloneqq\sum_{n\geq 1} \frac{a_n}{n^s}
$$
extends analytically to the complex plane via the Mellin transform and satisfies a functional equation relating $s$ to $k-s$. The critical values occur at integers   $1\leq m\leq k-1$.  By the Mellin transform, the special values admit explicit integral expressions,  
\begin{equation}
    L(f,m)=\frac{(2\pi)^{m}}{(m-1)!}\int_0^{\infty}f(it)t^{m-1}dt,
\end{equation}
and are often referred to as periods of the form $f$. Deligne's conjecture  asserts that there exist complex periods $c_f^{\pm}$ such that
$\frac{L(f,m)}{(2\pi i)^mc_f^{\pm}}\in \Q(a_n)$. The numbers $c_f^{\pm}$ arise from the comparison between Betti and de Rham realizations of the weight $k-1$ motive attached to $f$.  From the perspective of the Bloch–Kato conjecture, the critical values $L(f,k-1)$ encodes arithmetic information about the associated motive and Galois representation. 
When $f$ has complex multiplication (CM) by an imaginary quadratic field $K=\Q(\sqrt{d_K})$ with $d_K$ a fundamental discriminant, results of Shimura and Blasius \cite{shimuraperiod, blasius} show that one may choose canonical periods $c_f^{\pm}$ that can be written in terms of  Gauss sums and the Chowla-Selberg period $\Omega_{d_K}$  \cite{ChowlaSelberg, Zagier-modularform}.  The period $\Omega_{d_K}$ is given by 
$$
\Omega_{d_K}=\sqrt{\pi}\prod_{j=1}^{|d_K|-1}\G\left(\frac j{|d_K|}\right)^{\chi_{K}(j)w_K/4h_K},
$$
where $w_K$ is the order of the unit group in $K$, $\chi_{K}$ is the Kronecker character associated with $K$,  $h_K$ is the class number of $K$, and $\G(\cdot)$ is the gamma function.  

As a classical  example, the modularity theorem associates to every elliptic curve $E/\Q$ a weight~two Hecke eigenform $f_E$ such that 
\(
L(E/\Q, s) = L(f_E, s).
\)
Moreover, the value $L(f_E,1)=L(E,1)$ is  a rational multiple of the real period of $E$. When $E$ has CM  by an imaginary quadratic field~$K$, its $L$-function can be described more explicitly in terms of a Hecke character $\psi_E $ of $K$, which defines a one-dimensional Galois representation of $\Gal(\overline{K}/K)$. The Hecke eigenform $f_E$  arises via the induction of this representation to $G_\Q:=\Gal(\overline{\Q}/\Q)$\cite{DeligneSerre1974}. This perspective naturally leads to higher-weight Hecke eigenforms obtained from powers of Hecke characters.  More generally, one expects that $L$-functions arising from the  \'etale cohomology of higher-dimensional varieties defined over $\Q$  are modular.
In particular, certain $K3$ surfaces defined over~$\Q$ are known to be modular and correspond to weight~$3$ modular forms, see  \cite{AOP, voight2025s, elkies2013modularformsk3surfaces}  for example.

\subsection{Main results}
In this paper, we compute special $L$-values attached to a class of weight $3$ CM  Hecke eigenforms arising from hypergeometric motives. Our approach combines Ramanujan’s theory of elliptic functions to alternative bases with the modularity of CM elliptic curves over real quadratic fields.  Ramanujan’s alternative bases relate modular forms to hypergeometric series, have led to  many applications, including Ramanujan $1/\pi$-series, and have more recently played a role in the study of certain \emph{hypergeometric Galois representations} \cite[et al.]{BCM, DPVZ,  HLLT, Katz, Katz96,  LLL-low}, which are expected to be modular or automorphic 
according to the Langlands philosophy \cite{LaurentLangland, Langlands, NewtonLanglandsModularity}.   Our focus is on the representations of $G_\Q$ arising from hypergeometric type $K3$ surfaces \cite{AOP, HGK3}, and  tensor products $ \widetilde{E}_d(z)\otimes\widetilde{E}_d(1-z)$ from the following  families of elliptic curves
 \begin{align}\label{eq:hgellipticcurves}
      \widetilde{E}_2(z)&\colon y^2=x(1-x)(x-z), \qquad      
      \widetilde{E}_3(z)\colon y^2+xy+\frac{z}{27}y=x^3,\\
      \widetilde{E}_4(z)&\colon y^2=x\left(x^2+x+\frac{z}{4}\right), \qquad   
      \widetilde{E}_6(z)\colon y^2+xy=x^3-\frac{z}{432}.
  \end{align}
  
By combining Ramanujan’s alternative bases, moduli interpretations for modular curves, the modularity of  CM elliptic curves over real quadratic fields, and the relation between $ \widetilde{E}_d(z)$ and $ \widetilde{E}_d(1-z)$, we obtain explicitly special  $L$-values of the associated weight $3$ Hecke eigenforms in terms of hypergeometric evaluations.  The geometric source of these forms suggests that their periods factor as norms of CM periods, while the hypergeometric realization predicts expressions in terms of special hypergeometric values.  Conceptually, the periods of the weight $3$ form  coincides with a hypergeometric period, and  Clausen formulas identify the  tensor square with a rigid hypergeometric motive. Our result makes this relationship explicit and builds on earlier work \cite{SRiM, damerell1, damerell2,  LLL-low}.  It is explicitly stated below.  See \cite[et al.]{Beukers, Beukers-MF, Ito-Jacobi, mccarthy-period, Otsubo-HeckeL,  ENRosenK1, Zudilin-arithmetic, Zudilin-periods} for relevant discussions.

\begin{theorem}\label{thm:modularity} 
Let $t \in \Q$ be such that $z=\frac{1-\sqrt{1-t}}2\in \R$. Fix $d\in\{2,3,4,6\}$.  If the elliptic curve $\widetilde E_d\left(z\right)$ given in  \zcref{eq:hgellipticcurves} has complex multiplication,
then  there exists a  unique  normalized Hecke eigenform $f_{d,t}$ of weight $3$, 
for which we have an explicit number $C_{d,t} \in \Q$  such that 
\[
C_{d,t}\cdot L\!\left(f_{d,t},2\right)
= \,\pi^{2}\;
{}_3F_2\!\left[\begin{matrix}\tfrac{1}{2}&\tfrac{1}{d}&1-\tfrac{1}{d}\\[4pt]&1&1\end{matrix};\,t\right] \coloneqq  \,\pi^{2}\;\sum_{n\geq 0}\frac{\left(\frac 12\right)_n\left(\frac 1d\right)_n\left(\frac {d-1}d\right)_n}{n!^3} t^n,
\]
where $(a)_0=1$ and $(a)_n=a(a+1)\cdots(a+n-1)$. A complete list of such pairs $(d,t)$ 
is provided in \zcref{tab:main} and \zcref{append:all}, as well as the values $C_{d,t}$ and $ L\!\left(f_{d,t},2\right)$. The associated modular forms $f_{d,t}$  are identified by their LMFDB  labels \cite{LMFDB}. 
\end{theorem}

For each $(d,t)$ satisfying the assumptions of \zcref{thm:modularity}, it corresponds to an $1/\pi$-series 
    \begin{equation}
        \sum_{n\geq 0}\frac{\left(\frac 12\right)_n\left(\frac 1d\right)_n\left(\frac {d-1}d\right)_n}{n!^3} (A_{d,t}n+1)t^n=\frac{\delta_{d,t}}{\pi},
    \end{equation}
    for suitable algebraic numbers $A_{d,t},\delta_{d,t}$ (see \cite[et al.]{J-P_Borwein,D-G_Chudnovsky, Ramanujan1/pi}), which reflects the Legendre relation from the curve $\widetilde{E}_d(z)$.
For $(d,t)=(2,1/64)$,   the corresponding Ramanujan $1/\pi$-series is 
$$
\sum_{n\geq 0}\left(\frac{42}5n+1\right)\frac{\left(\frac 12\right)_n^3}{n!^3} \left(\frac 1{64}\right)^n = \frac{16}5\cdot \frac 1\pi. 
$$
The associated modular form is $f_{112.3.c.a}$  and we obtain the special value
$$
8L(f_{112.3.c.a},2) = \pi^{2}\;
\pFq32{\frac 12&\frac 12&\frac12}{&1&1}{\frac 1{64}} = \frac 27 \Omega_{-7}^2. 
$$
See also the  work of McCarthy \cite{mccarthy-period}. 

\begin{remark}\label{remark:main}
The numbers $t$ and $C_{d,t}$ appearing in \zcref{thm:modularity} 
arise as values of suitable modular functions evaluated at CM points. 
For $d=2,3,4,$ and $6$, the values $t$ correspond to the pairs of elliptic curves  $(\widetilde E_d(z), \widetilde E_d(1-z))$ and can be interpreted as CM-values  of modular curves arising from  $\Gamma_0(2)$, $\Gamma_0(3)^+$, $\Gamma_0(2)^+$, and $\Gamma_0(1)\coloneqq\SL_2(\Z)$, respectively (see \zcref{sec:modofhgr} or \zcref{tab:haupymodul}). These groups are arithmetic triangle groups commensurable with $\SL_2(\Z)$ \cite{Takeuchi-triangle, Takeuchi-classify}. 
For a fixed $d$, the number $t$ is uniquely determined by the  integrality of elliptic $j$-values and a CM point $\tau_d$ chosen in the preferred fundamental domain shown in \zcref{fig:fundamentaldomains}. In fact, the value $t$ can be taken as a rational CM-value of the hauptmodul $t_d$ for the  triangle group $(2, e_d ,\infty)$  taking values $0$, $1$, and $\infty$ at the elliptic points of order $\infty$ (namely, a cusp), $2$, and $e_d$, respectively \cite{Voight-book, Yang-modular}, where $e_d=2d/(d-2)$. 
\end{remark}

Some of the results or their variants have been discussed in the literature, such as \cite{EHMM2, Au, Ito, LLT,  Ono-book, Osburn-Straub,rogerslvalues, ENRosenK2}. 
Most existing approaches rely on integral representatives of hypergeometric series, moments of elliptic integrals, and their relationships with periods of elliptic curves. In addition, Osburn and Straub \cite{Osburn-Straub} showed that certain critical $L$-values can be interpreted in terms of special values of interpolated sequences. Our method is closest in spirit to those used in \cite{LLT, Osburn-Straub, Tu-Yang-Hypergeometric}.  

By exploiting algebraic transformations of ${}_3F_2$-functions, we obtain further explicit evaluations for the case of $d=6$, leading to the following result.

\begin{corollary}\label{cor:d6} Set
\begin{equation}
    \tau_D=\begin{cases}
        \sqrt{-D/4}, &\text{for } D= 8, 12, 16, 28,\\
        \frac{1+\sqrt{-D}}{2}& \text{for } D= 7, 11, 19, 27, 43, 67, 163.
    \end{cases}
\end{equation} 
For $D\neq 28$, the unique weight $3$ Hecke eigenform $f_{D}$  of level $D$, that has rational Fourier coefficients  and CM by $\Q(\sqrt{-D})$,  satisfies
           $$
                   \pi^2\cdot \pFq32{\frac 12&\frac16&\frac56}{&1&1}{\frac{1728}{j(\tau_{-D})}}  = a_D L(f_{D},2),
                 $$
                 for an explicit algebraic number $a_D$ listed in the  \zcref{tab:main-b},   where $j$ is the elliptic $j$-invariant. 
                   For $D=28$, we have 
                  $$
                   \pi^2\cdot \pFq32{\frac 12&\frac16&\frac56}{&1&1}{\frac{64}{614125}}  = \frac{\sqrt{3\cdot 5\cdot 7\cdot 17}}2 L(f_{7},2). 
                 $$
                
        \end{corollary}

\begin{remark}
    The value $a_D$ arises from CM values of the quotient of the level-$1$ Eisenstein series $E_4E_2^\ast/E_6$ and the factor $\sqrt{1-\frac{1728}{j}}$, and is only guaranteed to be in $\sqrt{1-\frac{1728}{j}} \cdot \Q$, by the work of Shimura \cite{Shimura-special-zeta}. This is because that the weight-$3$ cusp form $f_{d,t}$ associated with the relevant
    hypergeometric Galois representations (see \zcref{sec:Galois}) differ from those  $f_D$ considered here by the quadratic character $\chi_{1-t}:=\left(\frac{1-t}\cdot\right)$, namely $f_{d,t}=f_D\otimes \chi_{1-t}$, where $t=\frac{1728}{j(\tau_{-D})}$.\footnote{In these cases the associated modular forms via modularity have relatively large levels (see \zcref{Lemma: twist}) and are not currently listed in the LMFDB; for this reason, we do not identify the corresponding Hecke eigenforms.}

For example,
\[
  6\sqrt{5}\,L(f_{12.3.c.a},2)
  = \pi^2 \cdot \pFq32{\frac12&\frac16&\frac56}{&1&1}{\frac{4}{125}}
  = \frac{25}{4}\,L(f_{300.3.g.b},2).
\]
Indeed, the modular form $f_{300.3.g.b}$ is the quadratic twist of $f_{12.3.c.a}$ by $\chi_5$, 
whose Gauss sum is $\sqrt{5}$.
\end{remark}

\subsection{Summary and outline}
 
 Our method exploits the arithmetic of CM  elliptic curves $\widetilde{E_d}(z)$ and $\widetilde{E_d}(1-z)$  to identify the underlying motive as a tensor product. This perspective clarifies the role of the Deligne–Shimura–Blasius period and explains naturally why the resulting critical value coincides with a hypergeometric period, 
 reflecting the rigid  structure of the associated Galois hypergeometric representations. The resulting formulas provide an explanation for several previously observed identities and produce evaluations of $L$-values of certain CM modular forms and special ${}_3F_2$-values.  For example, from the proofs and our results, we have 
$$
\pFq32{\frac 12&\frac 13&\frac 23}{&1&1}1=\frac 3{\sqrt 5}\pFq32{\frac 12&\frac 16&\frac 56}{&1&1}{\frac{4}{125}}=\frac32\pFq32{\frac 12&\frac 14&\frac 34}{&1&1}{-\frac{16}{9}}=\frac{\sqrt[3]2\sqrt{3}}{4}\frac{\Omega_{-3}^2}{\pi^2}, 
$$
\begin{align*}
L(f_{112.3.c.a},2) = &\sqrt 7L(f_{7.3.b.a},2) = \frac 1{28}\Omega_{-7}^2 \\
= & \frac{\pi^2}{2\sqrt{15}}\pFq32{\frac 12&\frac 16&\frac 56}{&1&1}{-\frac{64}{125}} = \frac{2\pi^2}{\sqrt{255}} \cdot \pFq32{\frac 12&\frac16&\frac56}{&1&1}{\frac{64}{614125}},
\end{align*}
and
$$L(f_{144.3.g.a},2)=\frac{8\sqrt 3}9L(f_{16.3.c.a},2) = \frac{1}{12\sqrt{3}}\Omega_{-4}^2.   $$

The rest of the paper is organized as follows. 
In \zcref{sec:HG}, we review the necessary background on hypergeometric series and their relation to the elliptic curves $\widetilde{E}_d(z)$. 
In \zcref{sec:MF}, we recall basic facts on modular forms and Ramanujan’s alternative bases. 
In \zcref{sec:Galois}, we discuss the relevant hypergeometric Galois representations. 
In \zcref{sec:proof}, we give explicit Fourier expansions and determine the constants $C_{d,t}$ in \zcref{tab:main-a} and \zcref{tab:main-b}.
The appendix contains the relevant fundamental domains and necessary CM data,  as described in   \zcref{remark:main}. 

\begin{remark}\label{rmk:final}
    The approach developed here extends naturally to evaluation of certain $L$-values of CM forms of higher weight, as well as Rankin-Selberg convolution and  Petersson inner product of such CM forms; see, for example,  \cite[Section 11]{book:cohen-stromberg}, \cite{Zagier-modularform}, or \cite[Appendix]{Hsu-Tu-Yang}. 
    We conclude \zcref{sec:proof} by presenting two examples in each direction.
\end{remark}

\renewcommand{\arraystretch}{1.5}
\begin{table}[h!]
\centering

\hspace{-2cm}\begin{subtable}[t]{0.6\textwidth}
\centering
\scalebox{1}{
\begin{tabular}{|c||c|c|c|c|c|}
\hline
$d$ & $t$ & $\tau_d$ & $f_{d,t}$ & $C_{d,t}$ & $L(f_{d,t},2)$ \\ \hline
6 & 1 & $i$ & $f_{144.3.g.a}$ & 9 & $\frac{\sqrt 3}{36}\Omega_{-4}^2$\\
  & $\tfrac{4}{125}$ & $\sqrt{3}i$ & $f_{300.3.g.b}$ & $\tfrac{25}{4}$ & $\frac{\sqrt[3]2\sqrt{15}}{75}\Omega_{-3}^2$\\ \hline
2 & 1 & $\frac{1+i}{2}$ & $f_{16.3.c.a}$ & 16 & $\frac{1}{32}\Omega_{-4}^2$ \\
  & $-8$ & $\frac{i}{2}$ & $f_{16.3.c.a}$ & 8 & $\frac{1}{32}\Omega_{-4}^2$\\
  & $-1$ & $\tfrac{i}{\sqrt{2}}$ & $f_{32.3.d.a}$ & 8 &$\frac{1}{32}\Omega_{-8}^2$ \\
    & $\tfrac{1}{4}$ & $\tfrac{1+\sqrt{3}i}{2}$ & $f_{48.3.e.a}$ & 8 &  $\frac{\sqrt[3]2}{24}\Omega_{-3}^2$\\
  & $-\tfrac{1}{8}$ & $i$ & $f_{64.3.c.a}$ & 8 &$\frac{\sqrt 2}{32}\Omega_{-4}^2$ \\
  & $\tfrac{1}{64}$ & $\tfrac{1+\sqrt{7}i}{2}$ & $f_{112.3.c.a}$ & 8 &  $\frac1{28}\Omega_{-7}^2$\\ \hline
3 & 1 & $\tfrac{i}{\sqrt{3}}$ & $f_{12.3.c.a}$ & 18 &  $\frac{\sqrt[3]2\sqrt{3}}{72}\Omega_{-3}^2$\\
  & $\tfrac4{125}$ & $\tfrac{\sqrt{15}i}{3}$ & $f_{15.3.d.b}$ & $\tfrac{45}{4}$ &$\frac{\sqrt{3}}{90}\Omega_{-15}^2$ \\
  & $\tfrac12$ & $\tfrac{i\sqrt6}{3}$ & $f_{24.3.h.b}$ & 9 &$\frac{\sqrt{3}}{72}\Omega_{-24}^2$ \\
  & $-\tfrac{9}{16}$ & $\tfrac{1+\sqrt{3}i}{2}$ & $f_{27.3.b.a}$ & 9 &  $\frac{2\sqrt{3}}{81}\Omega_{-3}^2$\\ \hline 
4 & $-\tfrac{256}{3^4 7^2}$ & $\tfrac{1+\sqrt{7}i}{2}$ & $f_{7.3.b.a}$ & 21 & $\frac{\sqrt{7}}{196}\Omega_{-7}^2$ \\ 
& 1 & $\tfrac{i}{\sqrt{2}}$ & $f_{8.3.d.a}$ & 24 & $\frac{\sqrt 2}{96}\Omega_{-8}^2$ \\
  & $\tfrac{1}{7^4}$ & $\tfrac{3i}{\sqrt2}$ & $f_{8.3.d.a}$ & $\tfrac{56}{3}$ &  $\frac{\sqrt 2}{96}\Omega_{-8}^2$\\ 
  & $\tfrac{32}{81}$ & $i$ & $f_{16.3.c.a}$ & 12 & $\frac{
1}{32}\Omega^2_{-4}$\\
  & $\tfrac19$ & $\tfrac{i\sqrt6}{2}$ & $f_{24.3.h.a}$ & 12 &$\frac{\sqrt{6}}{144}\Omega_{-24}^2$ \\\hline
\end{tabular}}
\caption{ }
\label{tab:main-a}
\end{subtable}
\hspace{-0.5cm}
\begin{subtable}[t]{0.36\textwidth}
\centering
\scalebox{1}{
\begin{tabular}{|c|c|c|c|}
\hline
$D$ & $f_D$ & $a_D$ &$L(f_{D},2)$ \\ \hline
7 & $f_{7.3.b.a}$ & $2\sqrt{7\times 15}$& $\frac{\sqrt{7}}{196}\Omega_{-7}^2$\\ 
8 & $f_{8.3.d.a}$ & $6\sqrt{10}$ &  $\frac{\sqrt 2}{96}\Omega_{-8}^2$\\ 
11 & $f_{11.3.b.a}$ & $3\sqrt{22}$ & $\frac{2\sqrt{11}}{3\cdot 11^2}\Omega_{-11}^2$ \\ 
12 & $f_{12.3.c.a}$ & $6\sqrt{5}$& $\frac{\sqrt[3]2\sqrt{3}}{72}\Omega_{-3}^2$\\ 
16 & $f_{16.3.c.a}$ & $2\sqrt{33}$& $\frac{
1}{32}\Omega^2_{-4}$\\ 
19 & $f_{19.3.b.a}$ & $\sqrt{19\times 6}$ &  $\frac{2\sqrt{19}}{19^2}\Omega_{-19}^2$ \\ 
27 & $f_{27.3.b.a}$ & $3\sqrt{10}$& $\frac{2\sqrt{3}}{81}\Omega^2_{-3}$\\ 
43 & $f_{43.3.b.a}$ & $\tfrac{\sqrt{43\times 15}}3$& $\frac{12\sqrt{43}}{43^2}\Omega_{-43}^2$  \\ 
67 & $f_{67.3.b.a}$ & $\frac{\sqrt{67\times 330}}{19}$ & $\frac{38\sqrt{67}}{67^2}\Omega_{-67}^2$ \\ 
163 & $f_{163.3.b.a}$ & $\frac{\sqrt{163\times 10005}}{181}$& $\frac{724\sqrt{163}}{163^2}\Omega_{-163}^2$  \\ \hline
\end{tabular}}
\caption{ }
\label{tab:main-b}
\end{subtable}

\caption{ \centering (A) Modularity data, (B) $a_D$ values.}
\label{tab:main}
\end{table}
\medskip

 \begin{paragraph}{\bf Acknowledgments} The authors would like to express their gratitude to Ling Long, Esme Rosen, and Xingting Wang for their helpful discussions and suggestions.  Mondal was  supported as a graduate researcher  by NSF Grant DMS 2302514 during the Fall 2025 semester. Nakagawa was supported by JST FOREST Program (JPMJFR2235), and was hosted by the Department of Mathematics at Louisiana State University in the spring of 2025 while working on this project. Tu is partially supported by  NSF Grant DMS 2302531. 
 \end{paragraph}

\section{Hypergeometric functions} \label{sec:HG}
This section provides a summary of classical hypergeometric functions, covering their definition, key properties, and certain transformation identities. 
Throughout, all classical hypergeometric identities are assumed to hold whenever both sides of the equality are well defined and convergent. For detailed discussions, the reader is referred to \cite{andrews, Slater,  Yoshida-hypergeometric}. 

For multi-subsets $\alpha=\{a_1,\cdots,a_n\}$,  $\beta=\{b_1\coloneqq1,\cdots ,b_n \}$ of $\Q$ with $b_i\not \in \Z_{\leq 0}$, and $z\in \C$,  define the classical hypergeometric function ${}_{n}F_{n-1}$ as 
\begin{align}\label{definition:hgf}
	\pFq{n}{n-1}{a_1 &a_2&\ldots& a_n}{&b_2&\ldots&b_{n}}{z}=\sum_{k=0}^\infty\frac{(a_1)_k\cdots(a_n)_k}{(b_2)_k\cdots (b_n)_k}\frac{{z^k}}{k!}, 
\end{align}
where $(a)_0=1$, $(a)_k=a(a+1)\cdots(a+k-1)=\G(a+k)/\G(a)$  is the Pochhammer symbol and $\Gamma(\cdot)$  is the gamma function. The function ${}_nF_{n-1}$ converges when $|z|<1$. 
Each generalized hypergeometric function ${}_{n}F_{n-1}$ satisfies a linear
Fuchsian differential equation of order $n$ with three regular singular
points at $0$, $1$, and $\infty$. In the case $n=2$, the Gauss hypergeometric
function $\pFq{2}{1}{a&b}{&c}{z}$, regarded as a function of $z$, satisfies the
hypergeometric differential equation
\begin{equation}\label{eq:hde}
HDE(a,b,c;z):\qquad 
z(1-z)\frac{d^{2}F}{dz^{2}}
+\bigl[c-(a+b+1)z\bigr]\frac{dF}{dz}
-abF=0.
\end{equation}

\subsection{Algebraic transformation formulas} \label{sec: hgtransformations}
We record some algebraic transformation of hypergeometric functions.  We emphasize that the arguments on the both hand-sides of the identities are on a certain range. Similar identities can be obtained through the analytic continuations of the solutions of the corresponding $HDE$'s.  

{\bf Pfaff transformations}:
\begin{align}\label{pfaff}
     \pFq{2}{1}{a&b}{&c}{z}&=(1-z)^{-a}\pFq{2}{1}{a&c-b}{&c}{\frac{z}{z-1}}=(1-z)^{-b}\pFq{2}{1}{c-a&b}{&c}{\frac{z}{z-1}}.
\end{align}

{\bf Kummer quadratic transformations}: 

\begin{equation}\label{kummerquad1} 
      \pFq{2}{1}{a& b}{& a-b+1}{z}= (1-z)^{-a}\pFq{2}{1}{\frac{a}{2}& \frac{a+1}{2}-b}{&a-b+1}{\frac{-4z}{(1-z)^2}}.
 \end{equation}

 \begin{equation}\label{kummerquad} 
      \pFq{2}{1}{a& b}{& a-b+1}{z}= (1+z)^{-a}\pFq{2}{1}{\frac{a}{2}& \frac{a+1}{2}}{&a-b+1}{\frac{4z}{(z+1)^2}}.
 \end{equation}

{\bf Bailey cubic transformations}:  
 \begin{align}\label{baileycubic2f1}
     \pFq{2}{1}{a& \frac{1-a}{3}}{&\frac{4a+5}{6}}{z}= (1-4z)^{-a}\pFq{2}{1}{\frac{a}{3}& \frac{a+1}{3}}{&\frac{4a+5}{6}}{-\frac{27z}{(1-4z)^3}}.
 \end{align}

{\bf Goursat's degree $4$  transformations}: 
\begin{equation}\label{eq:trans:quartic}
 \pFq{2}{1}{ \frac{4a}{3} & \frac{4a+1}{3}}{&\frac{4a+5}{6}}{z}=(1+8z)^{-a} \pFq{2}{1}{ \frac{a}{3} & \frac{a+1}{3}}{&\frac{4a+5}{6}}{\frac{64z(1-z)^3}{(1+8z)^3}}. 
\end{equation}

The classical {\bf Clausen} formula  gives solution to certain symmetric square of $HDE(a,b,c;z)$:

\begin{align}\label{eq:clausen}
    \pfq{2}{1}{1-a&a}{&1}{z}^2&=\pfq{2}{1}{\frac{1-a}2&\frac a2}{&1}{4z(1-z)}^2=\pfq{3}{2}{\frac 12&1-a&a}{&1&1}{4z(1-z)}.
\end{align}

Using the above  transformations, we get the following identities. 

\begin{proposition}\label{prop:hgtransformations}

    The following equalities hold for  arguments in certain subsets of the unit disk: 
    \begin{align}
 \pFq{3}{2}{\frac{1}{2}&\frac{1}{6}&
\frac{5}{6}}{&1&1}{\frac{-27z}{(1-4z)^3}}&=(1-4z)^{\frac{1}{2}}\pFq{3}{2}{\frac{1}{2}&\frac{1}{2}&
\frac{1}{2}}{&1&1}{z}
=\left(\frac{1-4z}{1-z}\right)^{\frac{1}{2}}\pFq{3}{2}{\frac{1}{2}&\frac{1}{4}&
\frac{3}{4}}{&1&1}{\frac{-4z}{(z-1)^2}},
    \end{align}
    and 
\begin{equation}
       \pFq{3}{2}{\frac{1}{2}&\frac{1}{6}&
\frac{5}{6}}{&1&1}{\frac{64z(1-z)^3}{(1+8z)^3}}=  (1+8z)^{\frac{1}{2}}\pFq{3}{2}{\frac{1}{2}&\frac{1}{3}&
\frac{2}{3}}{&1&1}{4z(1-z)}.
 \end{equation}
\end{proposition}

\subsection{Elliptic families}\label{section:hgellipticcurves} 

For $d\in\{2,3,4,6\}$, we consider the 
families of elliptic curves 
  \begin{align}
      \widetilde{E}_2(z)&\colon y^2=x(1-x)(x-z), \qquad      
      \widetilde{E}_3(z)\colon y^2+xy+\frac{z}{27}y=x^3,\\
      \widetilde{E}_4(z)&\colon y^2=x\left(x^2+x+\frac{z}{4}\right), \qquad   
      \widetilde{E}_6(z)\colon y^2+xy=x^3-\frac{z}{432}.
  \end{align}
For all parameters $z$ away from the discriminant locus, each curve is smooth and has a one–dimensional space of holomorphic differentials, generated by a $1$-form $\omega_z$. 
Fix a continuous family of nontrivial $1$-cycles $\gamma(z)$ in the homology group $H_1(\widetilde{E}_2(z), \Z)$. The associated period function

 $\Omega_d(z)=\int_{\gamma(z)}\omega_z$   

is a multivalued holomorphic function of $z$, and  satisfies a second–order Picard–Fuchs 
equation with regular singularities at the points where the fiber degenerates. Such differential equations can be derived algorithmically from the defining equation of the family using the Griffiths-Dwork method \cite{griffiths, dworkzeta}.
A direct computation shows that for each $d \in \{2,3,4,6\}$, the differential equation satisfied by $\Omega_d(z)$ 
are equivalent to the $HDE \left(1/d,1-1/d,1; z\right)$ in \eqref{eq:hde}; therefore, the periods can be written in terms of algebraic multiples of  $ \pi\pfq{2}{1}{\frac{1}{d}&\frac{d-1}{d}}{&1}{z}$  (see also \cite{beukersHeckman, book:koblitz,  Silverman-EC, zhouhessepencil}).

\begin{example}\label{ex:periode2}
    For $d=2$, take $\omega_z=\frac{dx}{y}$.  The period integrals  
    \begin{equation}
        \int_z^1 \omega_z \quad \mbox{and}  \quad  \int_1^\infty \omega_z 
    \end{equation}
    are solutions of  $HDE(1/2,1/2,1;z)$.  
    For a real number $z$ with $0 < z< 1$,
       \begin{equation}
      \int_1^\infty \omega_z= \int_1^\infty \frac{dx}{\sqrt{x(1-x)(x-z)}}=-i\int_0^1 \frac{dx}{\sqrt{x(1-x)(1-xz)}}= -\pi i\cdot\pfq{2}{1}{\frac{1}{2}&\frac{1}{2}}{&1}{z}.  
    \end{equation}

Similarly, 
    $
      \displaystyle \int_z^1 \omega_z = \pi\cdot\pfq{2}{1}{\frac{1}{2}&\frac{1}{2}}{&1}{1-z}. 
      $
\end{example}

\begin{example}\label{ex:periode4}
  For $d=4$,  $\frac{dx}{y}$ is a holomorphic differential $1$-form for $\widetilde{E}_4(z)$.  When $z\in \R$ with $0<z<1$,
    \begin{align*}
       \int_0^\frac{-1+\sqrt{1-z}}{2} &\frac{dx}{\sqrt{x\left(x^2+x+\frac{z}{4}\right)}} =\pi i\cdot \pfq{2}{1}{\frac{1}{4}&\frac{3}{4}}{&1}{z} =-\int_{-\infty}^{\frac{-1-\sqrt{1-z}}{2}}& \frac{dx}{\sqrt{x\left(x^2+x+\frac{z}{4}\right)}}. 
    \end{align*}
The quadratic formula \eqref{kummerquad}  can be interpreted as periods identity as follows. For a given real number $z$ with $0<z<1$ and $R(z)=\frac{4z}{(z+1)^2}$. Then the elliptic curves 
$\widetilde{E}_4(R(z))$ and $\widetilde{E}_2(z)$ are isomorphic over $\Q((1+z)^{1/2})$ by 
$$
\widetilde{E}_4(R(z)):\, Y^2=X(X^2+X+\frac{z}{(1+z)^2})\quad \longleftrightarrow \quad \widetilde{E}_2(z):\, y^2=x(1-x)(x-z),
$$   
$$
(X,Y)=(-x/(1+z), y/(1+z)^{3/2}). 
$$
Then 
 $$
   \Omega_2(z)\coloneqq\int_{1}^{\infty} \frac{dx}{\sqrt{x(1-x)(x-z)}}=-\pi i\cdot\pfq{2}{1}{\frac{1}{2}&\frac{1}{2}}{&1}{z},
$$
and, by a change of variable, 
\begin{align}
    \Omega_2(z)=&  -\frac{1+z}{(1+z)^{3/2}}\int^{-\infty}_{-\frac1{1+z}} \frac{dX}{\sqrt{X\left(X^2+X+\frac{R(z)}{4}\right)}}=\frac{-\pi i}{(1+z)^{1/2}}\cdot \pfq{2}{1}{\frac{1}{4}&\frac{3}{4}}{&1}{R(z)}. 
\end{align}
\end{example}

\medskip
\section{Modular Forms}\label{sec:MF}

 In this section, we recall basic definitions and properties of modular forms that will be useful for our later discussions. Most of the materials can be found in the references \cite{book:cohen-stromberg, DS-modularforms, LLT, Miyake, Zagier-modularform}.

\subsection{Eisenstein series}
 We begin with the following  general definition of Eisenstein series to take care of convergence in certain cases. 
\begin{definition}    For a positive even integer $k$ and a complex number $s$ satisfying $\Re(2s+k)>2$, we define the non-holomorphic Eisenstein series of weight $k$, denoted by $G_k(s,\tau)$, as 
    \begin{equation}
        G_k(s,\tau)\coloneqq\frac{1}{2}\sum_{(m,n)\in \Z^2\setminus\{0,0\}} \frac{1}{(m\tau+n)^k}\frac{\Im(\tau)^s}{|m\tau+n|^{2s}}, \quad \tau\in \H,
    \end{equation}
    where $\H\coloneqq\{\tau \in \C:\, \Im(\tau)>0\}$. 
    \end{definition}
   The above series converges uniformly on any compact subset of $\H$ when $\Re(2s+k)>2$, and is modular on $\SL_2(\Z)$. 

\begin{proposition}
 The Eisenstein series $G_2(s,\tau)$ is analytic for $\Re(s)>0$ and has analytic continuation to $s=0$. The analytic continuation at $s=0$ 
is almost holomorphic modular form of weight $2$ 
with  the  $q$-expansion
    \begin{equation}\label{eq:g2astfourier}
        G_2^{*}(\tau)\coloneqq \lim_{\substack{s \rightarrow0,\  \Re(s)>0}}G_2(s,\tau)=\frac{\pi^2}{6}-\frac{\pi}{2\Im(\tau)}-4\pi^2 \sum_{m \geq 1}
        \frac{mq^m}{1-q^m}
        , \quad q=e^{2\pi i \tau}.
    \end{equation}
    For $k\geq 4$, we define $G_k(\tau)\coloneqq G_k(0,\tau)$, which has
    the $q$-expansion 
    \begin{equation}
        G_k(\tau)\coloneqq G_k(0,\tau) 
        =\zeta(k)+\frac{(2\pi i)^k}{\Gamma(k)}\sum_{m \geq 1}
        \frac{m^{k-1}q^m}{1-q^m}, \quad q=e^{2\pi i \tau},
    \end{equation}
    where $\zeta(s)$ is the Riemann zeta function $\zeta(s)=\sum_{n\geq  1}\frac 1{n^s}$.
\end{proposition}

The normalized Eisenstein series are
\begin{equation}\label{eq: ektogk}
    E_k\coloneqq\frac{G_k}{\zeta(k)}, \quad k\geq 4; \quad E_2^\ast:=\frac6{\pi^2}G_2^\ast. 
\end{equation} For any $N\in \N$,  the function 
\begin{equation}
    G_{2,N}(\tau)\coloneqq NG_2^*(N\tau)-G_2^*(\tau)
\end{equation}
is a holomorphic modular form of weight $2$ on $\Gamma_0(N)$.

Another set of important series, the \emph{Jacobi theta functions}, 
are defined as rapidly converging Fourier series associated with unary quadratic forms:
\begin{equation}
    \theta_2(\tau)\coloneqq\sum_{n \in \Z}q^{(2n+1)^2/8}, \quad \theta_3(\tau)\coloneqq\sum_{n \in \Z}q^{n^2/2}, \quad \theta_4(\tau)\coloneqq \sum_{n \in \Z}(-1)^nq^{n^2/2}.
\end{equation}
The following identities relate some of the functions $G_{2,N}$ and theta series.
\begin{proposition} \label{prop: g22halfintegershift}
For $\tau \in \H $, we have 

    \begin{align*}
       &\frac{\pi^2}2\theta_3\left(2\tau\right)^4=G_{2,4}(\tau),\\
        &\frac{\pi^2}{2} \theta_4(2\tau)^4=2G_{2,4}(\tau)-3G_{2,2}(\tau)=4G_{2,2}(2\tau)-G_{2,2}(\tau),\\  
        &G_{2,2}(\tau\pm1/2) = G_{2,4}(\tau)-{2}G_{2,2}(\tau). 
    \end{align*}
\end{proposition}
\medskip

We also recall Eisenstein series  constructed from the lattice sums with congruence conditions:
 \begin{equation}
     \mathcal{G}^*_{k,(a;N)}(\tau)\coloneqq
     \frac{1}{2}\sum_{m,n\in \Z} \frac{1}{[(Nm+a)\tau+n]^k}, \quad N,k\geq 1, a\geq 0,
 \end{equation}
 where $(0,0)$ is excluded when $a=0$. 
For $k>2$, these are weight $k$  modular forms on $\Gamma_1(N)$.
 It is easy to see from the definition that 
 $$
 \mathcal{G}^*_{2,(a;N)}(\tau)= \mathcal{G}^*_{2,(N-a;N)}(\tau), \quad \mbox{and} \quad  \mathcal{G}^\ast_{2,(a,N)}(\ell\tau)=\mathcal{G}^\ast_{2,(\ell a,\ell N)}(\tau),
 $$
for any positive integer $\ell$. These observations lead to the following properties. 
\begin{proposition}\label{prop:curlyg2asttoclassicales}
For any $\tau \in \H$, we have 
\begin{align*}
 \mathcal{G}^*_{2,(1;2)}(\tau)=&2\mathcal{G}^*_{2,(1;4)}(\tau)=G_2^\ast(\tau)-G_2^\ast(2\tau),\\
   \mathcal{G}^*_{2,(1;3)}(\tau)=&\frac12 \left(G_2^\ast(\tau)-G_2^\ast(3\tau)\right).
\end{align*}   
\end{proposition}

 Using the Lipschitz summation formula, we can get the  $q$-expansion for $ \mathcal{G}^*_{k,(a;N)}$. See \cite[Theorem 2.6]{LLT}
 for details. We state the case of $k=2$ below. 
\begin{proposition}
\label{prop: G2astfourier}
For fixed $a$ and $N$, the function $ \mathcal{G}^*_{2,(a;N)}(\tau)$ has the $q$-expansion 
\begin{equation}
     \mathcal{G}^*_{2,(a;N)}(\tau)=-\frac{\pi}{2N\Im(\tau)}-2\pi^2\sum_{n>0}n\frac{q^{na}+q^{n(N-a)}}{1-q^{Nn}}.
\end{equation}

\end{proposition}

\bigskip

We now consider the action of the Fricke involution $ W_\ell $.  
For a positive integer $ \ell $, we take 
\begin{equation}
    W_\ell \;=\;
\begin{pmatrix}
0 & -1 \\
\ell & 0
\end{pmatrix} : \, \tau \mapsto -1/\ell \tau. 
\end{equation}
The operator $ W_\ell $ gives an involution on the space of modular forms on $ \Gamma_0(\ell)$.
Recall that if $f$ is a modular form of weight $k$ then  
\begin{equation}\label{eq:frickedefinition}
 (f \mid W_\ell)(\tau)
    \;=\;
    \ell^{k/2} (\ell \tau)^{-k}
    f(-1/\ell \tau),
    \qquad \tau \in \mathbb{H}.
\end{equation}

Using  the modularity of  $G_2^*$, one has the following identities.
\begin{proposition}\label{prop:G2astfricke}
     For $N$, $\ell \in \N$, we have 
     $$
      G_2^*(N\tau)\mid W_{\ell}=\frac{\ell}{N^2} G^*_2\left(\frac{\ell \tau}{N}\right), \quad \mbox{and} \quad  G_{2,N}(\tau)\mid W_{\ell}=-\frac{\ell}{N} G_{2,N}\left(\frac{\ell \tau}{N}\right).
      $$
\end{proposition}

\begin{proof}
    By applying $W_{\ell}$, we get
    \begin{align}
        G_{2,N}\mid W_{\ell}(\tau)&=\frac{1}{\ell\tau^2}\left[NG_2^*\left(-\frac{N}{\ell \tau}\right)-G_2^*\left(-\frac{1}{\ell \tau}\right)\right]=\ell \left[\frac{1}{N}G_2^*\left(\frac{\ell \tau}{N}\right)-G_2^*(\ell \tau)\right]\\
        &=-\frac{\ell}{N}G_{2,N}\left(\frac{\ell\tau}{N}\right).
    \end{align}
    This completes the proof. 
\end{proof}

\begin{proposition}\label{prop:g2ast13fricke}
 For $\ell \in \N$, 
we have 
\begin{equation*}
     (\mathcal{G}^*_{2,(1;2)}\mid W_{\ell}) (\tau)=-\frac{\ell}4\mathcal{G}^*_{2,(1;2)}\left(\frac{\ell\tau}{2}\right)+\frac{3\ell}{4}G^*_2\left(\ell\tau\right)=-\ell\mathcal{G}^*_{2,(1;2)}\left(\frac{\ell\tau}{2}\right)+\frac{3\ell}{4}G^*_2\left(\frac{\ell\tau}{2}\right),
\end{equation*}
\begin{equation*}
         (\mathcal{G}^*_{2,(1;3)}\mid W_{\ell}) (\tau)=-\frac \ell9\mathcal{G}^*_{2,(1;3)}\left(\frac{\ell\tau}{3}\right)+\frac{4\ell}{9}G^*_2\left(\ell\tau\right) =-\ell\mathcal{G}^*_{2,(1;3)}\left(\frac{\ell\tau}{3}\right)+\frac{4\ell}{9}G^*_2\left(\frac{\ell\tau}{3}\right),
\end{equation*}
and 

\begin{equation*}
     (\mathcal{G}^*_{2,(1;4)}\mid W_{\ell}) (\tau)= -\frac{\ell}4\mathcal{G}^*_{2,(1;4)}\left(\frac{\ell\tau}{2}\right)+\frac{3\ell}{8}G^*_2\left(\ell\tau\right)=-\ell \mathcal{G}^*_{2,(1;4)}\left(\frac{\ell \tau}{2}\right)+\frac{3\ell}{8}G_2^*\left(\frac{\ell \tau}{2}\right). 
\end{equation*}

\end{proposition}
\begin{proof} 
     Using 
     \zcref{prop:curlyg2asttoclassicales} and \zcref{{prop:G2astfricke}},
     we obtain
     \begin{align*}
    (\mathcal{G}^*_{2,(1;2)}\mid W_{\ell}) (\tau)&=
    G_2^\ast(\tau)\mid W_{\ell}-G_2^\ast(2\tau)\mid W_{\ell}={\ell} G^*_2\left(\ell \tau\right)-\frac{\ell}{4} G^*_2\left(\frac{\ell \tau}{2}\right)\\
    &=-\frac{\ell}4\mathcal{G}^*_{2,(1;2)}\left(\frac{\ell\tau}{2}\right)+\frac{3\ell}{4}G^*_2\left(\ell\tau\right),
\end{align*}
which can also be written as $-\ell\mathcal{G}^*_{2,(1;2)}\left(\frac{\ell\tau}{2}\right)+\frac{3\ell}{4}G^*_2\left(\frac{\ell\tau}{2}\right)$. 
Thus
$$
    (\mathcal{G}^*_{2,(1;4)}\mid W_{\ell}) (\tau)= -\frac{\ell}4\mathcal{G}^*_{2,(1;4)}\left(\frac{\ell\tau}{2}\right)+\frac{3\ell}{8}G^*_2\left(\ell\tau\right)=-\ell \mathcal{G}^*_{2,(1;4)}\left(\frac{\ell \tau}{2}\right)+\frac{3\ell}{8}G_2^*\left(\frac{\ell \tau}{2}\right). 
$$
Similarly,  for $\mathcal{G}^*_{2,(1;3)}$, we have 
\begin{align}\label{eq:curlygw3}
    2\mathcal{G}^*_{2,(1;3)}(\tau)\mid W_{\ell} &=G_2^\ast\mid W_{\ell}(\tau)-G_2^\ast\mid W_{\ell}(3\tau)={\ell} G^*_2\left(\ell \tau\right)-\frac{\ell}{9} G^*_2\left(\frac{\ell \tau}{3}\right)\\
    &=-2\ell\mathcal{G}^*_{2,(1;3)}\left(\frac{\ell\tau}{3}\right)+\frac{8\ell}{9}G^*_2\left(\frac{\ell\tau}{3}\right),
\end{align}
which gives the desired identity.
\end{proof}

\subsection{Twisting modular forms} 
We now introduce the twist of a modular form.  
\begin{definition} For a modular form $f(\tau)$ with Fourier expansion $\sum_{n \ge 0} a_n q^n$ and a Dirichlet character $\chi$, we define 
the twist $f\otimes \chi$ as 
\[
(f \otimes \chi)(\tau) \coloneqq \sum_{n \ge 0} a_n \chi(n) q^n.
\]  
\end{definition}
\begin{lemma}\label{Lemma: twist}
    Let $f$ be a modular form for $\Gamma_0(N)$ with Dirichlet character $\chi'$, and $\chi$ be another Dirichlet character. Then 
    \begin{enumerate}
    \item if $f$ is a cusp form, then so is $f\otimes \chi$; 
    \item the form $f \otimes \chi$ is of level $M$ with Dirichlet character $\chi'\cdot \chi^2$, where 
    $$
      M=\mbox{lcm} (N, \mbox{cond}{\chi'}\mbox{cond}{\chi}, \mbox{cond}{\chi}^2), 
    $$
    and $\mbox{cond}{\psi}$ is the conductor of the character $\psi$.
    \end{enumerate}
   
\end{lemma}

For convenience, throughout, we will use $\chi_{_\ell}$ to denote the Dirichlet character  
\( \displaystyle
\chi_{_\ell} = \left( \frac{\ell}{\cdot} \right). 
\)

\subsection{Modular equations and special values} 

In this subsection, we provide  special values of a given hauptmodul on $\G_0(6)$ and relations between Eisenstein series.

\begin{proposition}\label{prop: gamma06}
Let 
$$
 u_6(\tau)=\frac{6\eta(2\tau)\eta(6\tau)^5+\eta(\tau)^5\eta(3\tau)}{12\eta(2\tau)\eta(6\tau)^5+\eta(\tau)^5\eta(3\tau)}.
$$
Then 
we have the identities
\begin{align*}
\frac{G_{2,2}(3\tau)}{G_{2,2}(\tau)}&=\frac{-2u_6(\tau)^2-2u_6(\tau)+1}{6u_6(\tau)^2-6u_6(\tau)-3},\\
\frac{G_{2,3}(2\tau)}{G_{2,3}(\tau)}&=u_6(\tau)^2. 
\end{align*}
The functions $x=u_6(\tau)$ and $y=u_6(3\tau)$ satisfy the modular equation
\begin{equation}\label{eq: modular equation}
    16x^3y^3 - 12(x^3y+xy^3) + 12x^2y^2  - 5(x^3-y^3) - 3(x^2y - xy^2) - 6(x^2+y^2) + 6xy+ 2=0.
\end{equation}

Moreover, we have the following special values:
\begin{align*}
&u_6\left(\frac{3+\sqrt{-15}}6\right)=\frac{\sqrt 5}2, \quad u_6(i/\sqrt 2)=\frac15+\frac{3\sqrt 6}{10}, \\
&u_6(i)=\frac{3^{3/4}\sqrt 2-1+\sqrt3}4, \quad u_6(i/3)=\frac{3^{3/4}\sqrt 2+1-\sqrt3}4. 
\end{align*}
\end{proposition}

\begin{remark}
    Note that $u_6 = \frac{6\widetilde{u}+1}{12\widetilde{u}+1}$ is a hauptmodul for $\G_0(6)$, where 
    $$\widetilde{u}(\tau)=\eta(2\tau)\eta(6\tau)^5/\eta(\tau)^5\eta(3\tau)$$  is a known hauptmodul for $\G_0(6)$ (see \cite[Table $3$]{sebbarhaupt} for example) and takes values at cusps as follows.
    $$
       \begin{array}{c|cccc} \hline
         \mbox{cusp}& 0&1/3&1/2&i\infty\\\hline
         u_6(\tau)& 1/2&-1/2&-1&1\\\hline
       \end{array}
    $$The purpose we make this choice is to have simpler expressions for later use.  
    We also remark that the special value $u_6(i)$ can be obtained by the special values of $\eta$-functions (cf. \cite{ramanujannotebook3, ramanujannotebook5, cohen, LLT}). 
\end{remark}
\begin{proof}
A standard method to get relations between modular functions  is to compare the divisor functions and $q$-expansions of the target functions; see  \cite{book:koblitz} for example.  We omit the computations. 

To obtain the special value $u_6\left(\frac{3+\sqrt{-15}}6\right)$ and $u_6(i)$, we make use of the modular function  
$t=\frac{\eta(3\tau)^{12}}{\eta(\tau)^{12}}$ on $\G_0(3)$ and the relations
\begin{align*}
t&=-\frac{4u_6^3-3u_6-1}{27( 4u_6^3-3 u_6+1)},\\
j(3\tau)&=\frac{729t(\tau)^4+756t(\tau)^3+270t(\tau)^2+36t(\tau)+1}{t(\tau)^3}. 
\end{align*}

When $\tau=\frac{3+\sqrt{-15}}6$, the known $j$-value $j\left(\frac{3+\sqrt{-15}}2\right)=-85995\left(\frac{1+\sqrt5}2\right)-52515$ gives us that
$t\left(\frac{3+\sqrt{-15}}6\right)=\frac{\sqrt 5-3}{54}$ by identifying the exact root using the approximation through the $q$-expansion.  Further, we can determine the value  $u_6\left(\frac{3+\sqrt{-15}}6\right)=\frac{\sqrt 5}2$ using the relation between $t$ and $u_6$.

Similarly, we can use the special values $j(i/\sqrt 2)=8000$, $j(i)=1728$, and the relation between $j(\tau)$ and $j(3\tau)$ with the equation given above to have $t(i/\sqrt{2})=\frac1{729}(5+2\sqrt 6)$ and $t(i)=\frac1{243}(-3+2\sqrt 3)$. Hence, we have $u_6(i/\sqrt 2)=\frac1{10}(2+3\sqrt 6)$ and $u_6(i)=\frac{3^{3/4}\sqrt 2-1+\sqrt3}4$ as desired.   \zcref{eq: modular equation} gives the value $u_6(i/3)$. 
\end{proof}

\subsection{Relation between Hypergeometric functions and Modular forms} \label{sec:RB} In this section, we introduce certain Ramanujan alternative bases. 
 
 For a finite index subgroup $\G\subseteq \SL_2(\Z)$, let $t$ be a modular function on $\G$ and $f$ be a 
 meromorphic modular function of weight $k$ on $\G$.  Thus, we can write locally $f(\tau)=F(t(\tau))$ for some function $F$ and that $F$ satisfies a linear differential equation of order $k+1$ with algebraic coefficients (see \cite[Proposition 5.3.36]{book:cohen-stromberg} and \cite{Yang-modular}).  
 For example,   the 
 function $E_{2k}^{1/(2k)}$ with $k\geq 2$ satisfies a linear differential equation of order $2$ with three regular singularities, which is a second order hypergeometric differential equation. For instances,
 \begin{align}\label{eq:e4e6tohgf}
     E_4^{1/4}(\tau)=\pFq{2}{1}{\frac{1}{12}&\frac{5}{12}}{&1}{\frac{1728}{j(\tau)}}. 
 \end{align}    
 Applying Clausen formula \eqref{eq:clausen}, one has 
 \begin{align}
     E_4^{1/2}(\tau)=\pFq{3}{2}{\frac{1}{2}&\frac{1}{6}&\frac{5}{6}}{&1&1}{\frac{1728}{j(\tau)}}=\left(1-4t_2(\tau)\right)^{\frac{1}{2}}\pFq32{\frac12&\frac12&\frac12}{&1&1}{t_2(\tau)},
 \end{align}
 where $t_2(\tau)=-64\frac{\eta(2\tau)^{24}}{\eta(\tau)^{24}}$. The last identity is parallel to the identity \zcref{prop:hgtransformations} with $z=t_2(\tau)$.

Similar to the case of $\SL_2(\Z)$, the monodromy group  of the differential equation $HDE(1/d,1-1/d,1;t)$ in \eqref{eq:hde}
can be realized as the subgroups $\G_0(4), \G_0(3)$ and $\G_0(2)$ for $d=2,3$ and $4$ respectively. Applying Clausen formula and quadratic formula \zcref{eq:clausen}, we can express modular forms on the groups $\G_d$ listed in \zcref{tab:haupymodul} as ${}_3F_2$-series with the given hauptmodul $t_d$ \cite{EHMMI,  WIN5, Imin-1, Imin-3}, when $\tau$ in the fundamental domain provided in \zcref{fig:fundamentaldomains} and the series ${}_3F_2(t_d(\tau))$ converges.

	\begin{table}[h]
		\begin{center}
			\scalebox{1.0}{
				\begin{tabular}{c c c c }
					\toprule
					\makecell{$d$}& \makecell{$\G_d$} & \makecell{$t_d(\tau)$} & \makecell{ $\pfq{3}{2}{\frac{1}{2}&\frac{1}{d}&\frac{d-1}{d}}{&1&1}{t_d(\tau)}$}\\
					\midrule
					$2$ & $\G_0(2)$ & $-64\frac{\eta(2\tau)^{24}}{\eta(\tau)^{24}}$ & $\theta_4(2\tau)^4$\\

                    $3$ & $\G_0(3)^+$ & $\frac{108\eta(\tau)^{12}\eta(3\tau)^{12}}{(\eta(\tau)^{12}+27\eta(3\tau)^{12})^2}$ & $\frac{3}{\pi^2}G_{2,3}(\tau)$\\

                    $4$ & $\G_0(2)^+$ & $\frac{256\eta(\tau)^{24}\eta(2\tau)^{24}}{(\eta(\tau)^{24}+64\eta(2\tau)^{24})^2}$ & $\frac{6}{\pi^2}G_{2,2}(\tau)$\\

                    $6$ & $\G_0(1)$& $\frac{1728}{j(\tau)}$ & $E_4(\tau)^{1/2}$
					
                    \\ \bottomrule
					
			\end{tabular}}
			\caption{\centering Hypergeometric functions as modular forms 
			}
			\label{tab:haupymodul}
		\end{center}
	\end{table}

\medskip
\section{Galois Representations}\label{sec:Galois} 

In CM theory,
Hecke characters arise as fundamental building blocks.  In~\cite[et al.]{AOP, TuYang2018}, Hecke characters, together with the modularity of hypergeometric functions, are used to obtain explicit evaluations of hypergeometric functions in terms of Hecke character values.  Relevant discussions can be found in, for example, \cite{AO00, PA-KQ, SRiM,  BKS,  WIN2, LD-McCarthy,McCarthy-Splitting, Sulakashna-Barman}. 

In this section, we illustrate how the $\ell$-adic representations of $G_\Q\coloneqq\Gal(\overline{\Q}/\Q)$
attached to modular forms, CM elliptic curves, and hypergeometric data are interconnected. In particular, we see how their Frobenius traces encode the same arithmetic data, such as Hecke eigenvalues, values of Hecke characters, and finite-field hypergeometric sums—thereby providing explicit bridges between modularity, complex multiplication, and hypergeometric motives.

\subsection{Hecke Characters and Hecke $L$-functions}\label{heckecharacters} 
\sloppy Hecke characters, also identified as Gr\"ossencharaktere or id\'ele class characters for number fields,  generalize the notion of Dirichlet characters and have associated $L$-functions that also satisfy functional equations, following the treatment in \cite{Ireland-Rosen} and \cite{book:silverman}; other references include \cite{Neukirch, Raghuram-HeckeNotes}. In what follows, we provide a brief overview and present an example to demonstrate more explicit connections to the CM theory.

Let \(K\) be a number field  with ring of integers
\(\mathcal{O}_{K}\).  
A modulus \(\mathfrak{m}\) is a nonzero ideal of \(\mathcal{O}_{K}\).  Let \(I_{K}(\mathfrak{m})\) denote the group of fractional ideals of \(K\) that are
coprime to \(\mathfrak{m}\), and let \(P_{K,1}(\mathfrak{m})\) be the subgroup of
principal ideals \((\alpha)\) satisfying the congruence
$\alpha \equiv 1 \mod{\mathfrak{m}}$. 

A {Hecke character} modulo \(\mathfrak{m}\) is a group homomorphism
\[
\psi : I_{K}(\mathfrak{m}) \longrightarrow \C^{\times}
\]
that factors through the ray class group $\mathrm{Cl}_{\mathfrak{m}}\coloneqq I_{K}(\mathfrak{m}) \big/ P_{K,1}(\mathfrak{m})$, and is determined by
a pair \((\psi_{f}, \psi_{\infty})\) associated with the modulus
\(\mathfrak{m}\) and infinity type \( \{(n_{v})\}_{v \mid \infty}\), where \(n_{v} \in \Z\). 
The finite component is a finite-order character
\[
\psi_{f} : (\mathcal{O}_{K} / \mathfrak{m})^{\times} \longrightarrow \C^{\times},
\]
and for any \(\alpha \in P_{K,1}(\mathfrak{m})\), the archimedean component is defined by
\[
\psi_{\infty}(\alpha)
    \coloneqq \prod_{v \mid \infty} \sigma_{v}(\alpha)^{\,n_{v}}.
\]
For a principal ideal \((\alpha)\) with \((\alpha, \mathfrak{m}) = 1\),  the value of
\(\psi\) is given by
\[
\psi\big((\alpha)\big)
    \;=\;
      \psi_{f}(\alpha)\, \psi_{\infty}(\alpha).
\]
The {weight} of the Hecke character $\psi$ is $m=\sum_{v\mid\infty}n_v$.

Let $\psi$ be a Hecke character of conductor $\mathfrak{f}$, which is the minimum modulus such that $\psi$ factors through it. The associated Hecke $L$-function is defined 
by
\[
L(s, \psi) = \prod_{\mathfrak{p} \nmid \mathfrak{f}} \left( 1 - \frac{\psi(\mathfrak{p})}{N\mathfrak{p}^{\,s}} \right)^{-1}, \quad \Re(s) > 1,
\]
where $N\mathfrak{p}$ is the absolute norm of $\mathfrak{p}$. It admits meromorphic continuation to $\mathbb{C}$ and satisfies a functional equation relating $L(s,\psi)$ and $L(m+1-s,\psi)$. 
When the field $K$ is an imaginary quadratic field and $\psi$ is non-trivial, the converse theorem for modular forms \cite{book:winnieli, weilmodularity} implies that, there exists a cuspidal Hecke eigenform $f_{\psi}(\tau)$ 
of weight $m+1$ such that 
\begin{equation}
    L(f_{\psi},s)=L\left(s,\psi\right). 
\end{equation}
 
\medskip

We now recall the connection between CM elliptic curves and Hecke characters. For simplicity, we give the statement for curves have CM by maximal imaginary quadratic orders \cite{lang-CM, book:silverman}. 
\begin{theorem}[Deuring]\label{thm:Deuring-simple}
Let $E$ be an elliptic curve defined over a number field $F$  of conductor $\mathfrak{f}$ and having complex multiplication  by the ring of integers 
$\mathcal{O}_K$ of an imaginary quadratic field $K$.

\begin{enumerate}[(a)]
\item If $K \subset F$, then there is a Hecke character 

$$\psi_{E/F} : I_F(\mathfrak{f}) / K^\times \longrightarrow \mathbb{C}^\times
$$
such that
$$
L(E/F, s) \;=\; L(s, \psi_{E/F})L(s, \overline{\psi_{E/F}}).
$$

\item If $K \not\subset F$, let $L = F K$ and let

$$\psi_{E/L} : I_{L}(\mathfrak{f}) / K^\times \longrightarrow \mathbb{C}^\times
$$
be the Hecke character attached to $E/L$. Then
$$
L(E/F, s) \;=\; L(s, \psi_{E/L}).
$$
\end{enumerate}
Here $L(E/F, s)$ is the $L$-function of $E/F$.
\end{theorem}
We will only require part (b) of this result, since our elliptic curves are defined over totally real fields; however, we state the full result for completeness. The generalized treatment, including the case when $E$ has CM by a non-maximal order $\mathcal{O}$, can be found in \cite{rubin-silverberg} for example. 

\begin{remark}
For a $\mathrm{CM}$ elliptic curve \(E\), the associated Hecke character
\(\psi_{E}\) has weight \(1\) and corresponds to a modular form
of weight \(2\).  More generally, a CM modular form of weight \(k\) arises from a
Hecke character of weight \(k-1\). 
\end{remark}

\begin{example}\label{ex:EC32}
Consider the elliptic curve
\(
E:\; y^2 = x^3 - x
\)
defined over $\Q$, which has CM by  \(\Z[i]\). For a  prime \(p\geq 3\), let $\#E(\F_p)$ denote the number of $\F_p$-rational points on $E$. Then 
\[
\#E(\F_p) = p+1-a_p(E), 
\]
with 
\[
a_p(E)\coloneqq -\sum_{x\in \F_p}\Big(\tfrac{x^3-x}{p}\Big) = \begin{cases}
-\chi(-1) J(\chi,\chi)-\overline{\chi(-1) J(\chi,\chi)}, & p\equiv 1\mod{4}, \\[6pt]
0, & p\equiv 3\mod{4},
\end{cases}
\]
where \(\chi\) is a fixed order $4$ character on $\F_p$ and \(J(\chi,\chi)\) is the Jacobi sum \(J(\chi,\chi):=\displaystyle\sum_{a\neq 0,1}\chi(a)\chi(1-a)\).  
For a prime $p \equiv 1 \mod{4}$, write $p=\pi \overline{\pi}$ with 
\(\pi \equiv 1 \mod{(2+2i)}\).
Take
$
\chi \;\coloneqq\; \Big(\tfrac{\cdot}{\pi}\Big)_4
$
on the residue field at the prime ideal $(\pi)$, 
where 

$
\Big(\tfrac{x}{\pi}\Big)_4 \;=\; x^{\,\tfrac{N(\pi)-1}{4}} \mod{\pi}.
$
Then we have
\[
-\chi(-1)J(\chi,\chi)=\pi.
\]
The associated Hecke character \(\psi_E\colon S\to\C^\times\), defined on the ideals
$
S=\{\mathfrak{J}\subseteq\O_K : (\mathfrak{J},2)=1\},
$
for a prime ideal $\mathfrak{p}\in S $ satisfies
\[
\psi_{E}(\mathfrak{p})=
\begin{cases}
\pi, & \mathfrak{p}=(\pi)\quad\text{ for }p\equiv 1\mod 4 \\[6pt]
-p, & \mathfrak{p}=(p)\quad\text{ for }p\equiv3\mod 4.
\end{cases}
\]
The cusp forms  associated with the Hecke characters $\psi_E$ and $\psi_E^2$ are given by 
\begin{equation}
    f_{32.2.a.a}(\tau)=\sum_{\mathfrak{I}\in S}\psi_E(\mathfrak{I})q^{N(\mathfrak{I})}=\sum_{m,n\in\Z}(-1)^n(4m+1-2ni)q^{(4m+1)^2+4n^2},
\end{equation}
\begin{align}
    f_{16.3.c.a}(\tau)&=\sum_{\mathfrak{I}\in S}\psi_E^2(\mathfrak{I})q^{N(\mathfrak{I})}=\sum_{m,n\in\Z}(4m+1-2ni)^2q^{(4m+1)^2+4n^2}= \frac 12\sum_{m,n\in \Z}(m+2ni)^2q^{m^2+4n^2}.
\end{align}

\end{example}

\medskip

\begin{example}\label{example:e27} 
Consider the $\Q$-elliptic curve 
\[
E:   y^2+xy-\frac y{216} =x^3, 
\]
which is $\widetilde{E}_3(-1/8)$, has conductor $27$, CM by $\Z[\frac{1+\sqrt{-27}}2]$, and is $\Q$-isogenous to $y^2+y=x^3$.  
For a prime $p\equiv1\mod{3}$, we  write $p=\pi\overline{\pi}$ in $\Z[\zeta_3]$, where $\pi\equiv2\mod{3}$ is a primary prime. 
The associated Hecke character \(\psi_E\) satisfies
\[
\psi_E(\mathfrak{p}) \;=\; 
\begin{cases}
\pi, & \mathfrak{p}=(\pi),\quad \text{ for  } p \equiv1\mod 3 , \;\; \\ 
-\,p, & \mathfrak{p}=(p), \quad \text{ for  } p \equiv 2 \mod 3. \;\;
\end{cases}
\]
The cusp forms  associated with the Hecke characters $\psi_E$ and $\psi_E^2$ are given by 
\begin{equation}
    f_{27.2.a.a}(\tau) =\sum_{\mathfrak{I}\in S'}\psi_E(\mathfrak{I})q^{N(\mathfrak{I})}=\sum_{m,n \in \Z}(3m+2+3n \zeta_3)q^{|3m+2+3n \zeta_3|^2},
\end{equation}
\begin{equation}
     f_{27.3.b.a}(\tau) 
     =\sum_{m,n\in \Z}(3m+2+3n \zeta_3)^2q^{|3m+2+3n \zeta_3|^2}=\frac12\sum_{m,n\in \Z}(m+3n\zeta_3)^2q^{|m+3n \zeta_3|^2},
\end{equation}
where $S'$ is the set of integral ideals of $\Z[\zeta_3]$ which are coprime to $3$. 
\end{example}

\subsection{Hypergeometric  Galois representations}\label{sec:modofhgr}

For multi-sets $\alpha=\{a_1,\cdots,a_n\}$,  $\beta=\{b_1\coloneqq1, b_2,\cdots ,b_n\}$ of $\Q$, let $\HD = \{\alpha; \beta\}$ denote a hypergeometric datum, and let  
$M = \mathrm{lcd}(\alpha \cup \beta)$ be the least positive common denominator of the entries in $\alpha$ and $\beta$.  
The datum $\HD$ is said to be \emph{primitive} if $a_i - b_j \notin \mathbb{Z}$ for all pairs $(i,j)$. Let \(c \in (\mathbb{Z}/M\mathbb{Z})^{\times}\), 
 define 
\(
cHD \coloneqq \bigl(\{\, c\alpha \bmod \mathbb{Z} \,\},\; \{\, c\beta \bmod \mathbb{Z} \,\}\bigr).
\)
We say that \(HD\) is \emph{defined over \(\mathbb{Q}\)} if, for every such \(c\), datum satisfies
\(
HD \equiv cHD \mod{\mathbb{Z}}.
\)

For a finite field $\F_q$ 
with $M \mid (q - 1)$, one can define  finite-field hypergeometric function $H_q(\HD; \lambda)$ for $\lambda \in \F_q$ 
using character sums (see \cite{BCM, Katz, Katz96}).  
This construction gives rise to $\ell$-adic Galois representations. We state here the results with the data defined over $\Q$ and $b_i=1$ for our cases. 

\begin{theorem}[\cite{BCM, Katz, Katz96, LLT2}] \label{katz BCM}
Let  $\ell$ be a prime. Given a  primitive   datum  $HD\coloneqq\{\alpha=\{a_1,\cdots,a_n\},\,  \beta=\{1,\cdots, 1\}\}$ defined over $\Q$ with  $M = \rm{lcd}(\alpha)$, for any  $t \in \Z[1/M] \backslash \{0\}$, there exists an $\ell$-adic Galois representation 
$$ 
  \rho_{HD,t}: \, G_\Q\coloneqq\Gal(\ol \Q/\Q)\longrightarrow GL(W_{t}),
$$ unramified almost everywhere, such that at each prime $p$ of $ \Z[1/(M\ell t)]$,
\begin{equation*} 
   \mbox{Tr} \rho_{HD,t}(\text{Frob}_p)=  
   H_p(\alpha,\beta; 1/t)  \quad \in \Z,  
\end{equation*} 
where  
$\text{Frob}_p$ is the  geometric  Frobenius conjugacy class of $G_\Q$ at $p$. 
\begin{enumerate}
\item When $t\neq 1$, $\dim_{\overline \Q_\ell}W_{t} = n$. All roots of the characteristic polynomial of $\rho_{HD,t}(\Frob_p)$  are algebraic with the same absolute value $p^{(n-1)/2}$. Moreover, $W_t$ admits a symmetric (resp. alternating) bilinear pairing if $n$ is odd (resp. even). 
\item  When $t=1$, $\dim_{\overline \Q_\ell}W_{1} = n-1$.  Moreover, $\rho_{HD,1}$ has a subrepresentation $\rho^{prim}_{HD,1}$ of dimension  $2\lfloor \frac {n-1}2 \rfloor$; the representation space of $\rho^{prim}_{HD,1}$  admits a symmetric (resp. alternating) bilinear pairing if $n$ is odd (resp. even).  All roots of the characteristic polynomial of $\rho_{HD,\ell}^{prim}(\Frob_p)$ have absolute value $p^{(n-1)/2}$.  
\end{enumerate}
\end{theorem}

\medskip

	For each $d=2,3,4,6$, we consider  the following hypergeometric datum:
	$$
		 \mathrm{HD}_3(d,t)\coloneqq\left\{ \begin{array}{ccc} \frac12, &\frac1d, &1-\frac1d\\ 1,&1,&1\end{array}; t\right\}, \quad t \in \Q.
	$$
Then for a fixed pair $(d,t)$, by \zcref{katz BCM}, there exists a $3$-dimensional ($2$-dimensional when $t=1$) $\ell$-adic Galois
representation attached to the
datum $\HD_3(d,t)$,
\[
\rho_\ell \colon G_{\Q}\longrightarrow
GL_3(\overline{\Q}_\ell),
\]
unramified outside a finite set of primes (depending on $d$ and $t$), 
such that
for every good prime $p$, 
\[
\Tr\!\bigl(\rho_\ell(\Frob_p)\bigr)=H_p\!\bigl(\HD_3(d,t)\bigr).
\]
These representations can be further decomposed  using  $\F_q$-version of the Clausen formula \eqref{eq:clausen} (see \cite[Appendix]{SRiM} or \cite{LLL-low} for more details) and can be realized as part of the tensor product of the representations attached to the elliptic curves $\widetilde{E}_d(z)$ and $\widetilde{E}_d(1-z)$ arising from the datum 
$$
		 \mathrm{HD}_2(d,z)\coloneqq\left\{ \begin{array}{cc} \frac1d, &1-\frac1d\\ 1,&1\end{array}; z\right\},\quad  t=4z(1-z).
$$
Note that for $\l\in \F_q$, if $\widetilde{E}_d(\l)$ is an elliptic curve over $\F_q$,  we have
\[
\#\widetilde{E}_d(\l)(\F_q)=q+1-H_q\!\bigl(\HD_2(d,\l)\bigr)=q+1-\phi(k_d)H_q\!\bigl(\HD_2(d,1-\l)\bigr),
\]
where $\phi$ is the quadratic character on $\F_q$ and 
\[
k_2 = k_6 = -1,\qquad k_3=-3,\qquad k_4=-2.
\]

 We now divide the situation into two cases: For $z \in F$  with  $F=\Q$ or that $F$ is a real quadratic extension of $\Q$, to discuss the representation attached to $ \mathrm{HD}_3(d,t)$ with $t\neq 1$ and that $\widetilde{E}_d(z)$ has CM by a field $K=\Q(\sqrt{-D})$. The case $t=1$ can be interpreted similarly. 
\medskip
To ease the notation, for a fixed pair $(d,t)$, we set $z=\frac{1-\sqrt{1-t}}{2}$, and $\rho_{\mathrm{HD}_2,\ell}$  (resp.  $\rho_{\mathrm{HD}_3,\ell}$)  the $\ell$-adic representations of $G_F\coloneqq\Gal(\ol \Q/F)$  (resp. $G_\Q$) attached to $ \mathrm{HD}_2(d,z)$  (resp. $ \mathrm{HD}_3(d,t)$).  
\medskip

\paragraph{\bf Case 1: \(z \in \mathbb{Q}\) and $z\neq 1/2$ }
In this case, for a good prime $p$, the finite--field Clausen gives

\begin{equation}\label{FFClausen}
  H_p\!\left( \mathrm{HD}_3(d, t)
  \right)
  = 
  H_p\!\left(\mathrm{HD}_2(d, z)
  \right)^{\!2} 
  \;-\; p.
\end{equation}

Thus the hypergeometric trace on the left corresponds to the symmetric square  
of the \(2\)-dimensional \(GO\)-type Galois representation associated to \(\mathrm{HD}_2(d,z)\),  
\[
\rho_{\mathrm{HD}_2,\ell} \simeq \mathrm{Ind}_{G_K}^{G_\mathbb{Q}}(\psi),
\]
for a Hecke character \(\psi\) of \(K\). 
Consequently, as $G_\Q$ representations we have
\[
\rho_{\mathrm{HD}_3,\ell}
\simeq \, \mathrm{Sym}^2 (\rho_{\mathrm{HD}_2,\ell})
\simeq
\mathrm{Ind}_{G_K}^{G_\mathbb{Q}}(\psi^2)
\oplus 
\chi_{K}\cdot\epsilon_\ell,
\]
where $\epsilon_\ell$ is the cyclotomic character.   This yields
\[
\mathrm{Tr}\big(\rho_{\mathrm{HD}_3,\ell}(\mathrm{Frob}_p)\big)
=
\mathrm{Tr}\big( \mathrm{Ind}_{G_K}^{G_{\mathbb{Q}}}(\psi^2)(\mathrm{Frob}_p) \big)
\;+\;
{\left(\tfrac{-D}{p}\right)}p. 
\]

The representation $\mathrm{Ind}_{G_K}^{G_\mathbb{Q}}(\psi^2)$  is modular and corresponds to a weight--\(3\) Hecke eigenform $f=q+\sum_{n\geq 2}a_n(f) q^n$ with $a_p =\mathrm{Tr}\!\big( \mathrm{Ind}_{G_K}^{G_{\mathbb{Q}}}(\psi^2)(\mathrm{Frob}_p) \big)$.  
In other words,  if $N$ is the level of $f$, we have 
\begin{equation}\label{apCase1}
a_p(f)
=
H_p\!\left( 
\mathrm{HD}_3(d,t)
\right)
\;-\;
{\left(\tfrac{-D}{p}\right)}p, \quad p\nmid N.
\end{equation}

In  \zcref{example:e27}, the $p$th Fourier coefficient of the CM
newform $f_{27.3.b.a}$ is given by
\[
a_p\!\left(f_{27.3.b.a}\right)
=
H_p\!\left(
\left\{ \tfrac{1}{2}, \tfrac{1}{3}, \tfrac{2}{3} \right\},
\{1,1,1\};\, -\tfrac{9}{16}
\right)
\;-\;
\left(\tfrac{-3}{p}\right)p, \quad p>3.
\]
In terms of the Galois representation attached to the Hecke
character $\psi_E$ of $\Q(\sqrt{-3})$, we have
\[
a_p\!\left(f_{27.3.b.a}\right)
=
\Ind_{G_{\Q(\sqrt{-3})}}^{G_{\Q}}
\!\left(\psi_E^2\right)\!\left(\Frob_p\right)
=
\psi_E^2(\mathfrak{p}) + \psi_E^2(\overline{\mathfrak{p}}),
\]
where $\mathfrak{p}\mid p$ is a prime of $\Q(\sqrt{-3})$.

\medskip


\paragraph{\bf Case 2: \(z \in F\), a quadratic extension of $\Q$. } If $1-t$ is a square in $\F_p$ we have \zcref{FFClausen}; for $1-t$ non-square in $\F_p$, Clausen relation acquires a quadratic twist.  
A refined identity states 
\begin{equation}\label{FFClausen2}
H_p\!\left(HD_3(d,t)
\right)
\\[6pt]
=
\left(\tfrac{k_d}{p}\right)
H_{p^2}\!\left(
HD_2(d,z)
\right)
\;-\;
p, \quad \mbox{if } \left(\tfrac{1-t}{p}\right)=-1. 
\end{equation}

The expression on the left of \zcref{FFClausen2} corresponds to the symmetric square up to twist by a finite order character of a  \(2\)-dimensional Galois representation \cite{LLL-low}. 
Let $\psi_E$ be the Hecke character of \(FK\) associated to the CM elliptic curve \(E/F\). 
Similar to the first case, through the modularity, the associated weight \(3\) Hecke eigenform $f$ CM by $K$ satisfies  $a_p(f)=0$, for  $p$ inert in $K$, and 
\[
a_p(f)
=H_p(\mathrm{HD}_3(d,t))
\;-\;
\left(\tfrac{1-t}{p}\right)p   =\left(\tfrac{\kappa_d}p\right)\left(\psi_E(\mathfrak p)+\overline{\psi_E(\mathfrak p)}\right),
\]
for $p$ splits in $K$, and $\mathfrak p $ is a prime of $K$ lying above $p$, by viewing \(\psi_E\) as a character  of \(K\).
\begin{example}
Take \(d=2\). 
Consider the case CM-case of  $t_2(\sqrt{-2})=-1$ with $z=\frac{1-\sqrt{2}}{2}$.
The associated hypergeometric data are
\[
\mathrm{HD}_{2}:=\left\{ \begin{array}{cc} \frac12, &\frac 12\\ 1,&1\end{array};  \tfrac{1-\sqrt{2}}{2} \right\},
\qquad
\mathrm{HD}_{3}:=\left\{ \begin{array}{ccc} \frac12, &\frac 12, &\frac 12\\ 1,&1, &1\end{array};  1 \right\}.
\]

Let
\[
E:=\widetilde{E}_2(z):\; y^{2}=x(1-x)\Bigl(x-\tfrac{1-\sqrt{2}}{2}\Bigr).
\]
Then \(E\) is a CM elliptic curve defined over \(F=\Q(\sqrt{2})\) with CM by \(\Z[\sqrt{-2}]\), i.e., \(K=\Q(\sqrt{-2})\).  The isomorphism class labeled in LMFDB is \href{https://www.lmfdb.org/EllipticCurve/2.2.8.1/256.1/c/3}{$2.2.8.1-256.1-c3$}. 
Set
\[
L \coloneqq KF = \Q(\sqrt{2}, i)=\Q(\zeta_8).
\]
Let \(\psi_E\) denote the Hecke character of \(L\) associated to the CM elliptic curve
\(E/F\).

Suppose that $p$ is a good prime for the curve $E$. When $p\equiv \pm 1\mod 8$, we let  $p=\mathfrak{p}\mathfrak{p}^\sigma$ in \(F\), where $\sigma$ generates $\Gal(F/\Q)$. Further we let $\F_\wp$ be the residue field at $\wp$ and fix the choice of $\sqrt 2$ in $\F_\wp$ such that  $ p+1-\#E(\F_\wp)=H_p(\mathrm{HD}_2)$. If $p\equiv -1 \mod 8$, that is, $E$ is supersingular at $p$. Then  
$$
H_p(\mathrm{HD}_2)^2=0=H_p(\mathrm{HD}_3)+p. 
$$
If $p\equiv 1 \mod 8$, suppose that $\mathfrak B\ol{\mathfrak B} = \wp$ in $L$. Then 
$$
p+1-\#E(\F_\wp)=H_p(\mathrm{HD}_2) = \psi_E(\mathfrak B)+\overline{\psi_E(\mathfrak B)}
$$
and 
$$
H_p(\mathrm{HD}_3)-p= \psi_E^2(\mathfrak B)+\overline{\psi_E(\mathfrak B)}^2. 
$$
Therefore, for the weight $3$ modular form with $q$-expansion $f=q+\sum_{n\geq 2}a_n q^n$, we should have
$$
a_p(f)=H_p(HD_2)^2-\left(1+\left(\tfrac{-2}p\right)\right)p.
$$
Similarly, when $p\equiv \pm 3\mod 8$, we let $\mathfrak B\ol{\mathfrak B} = p$ in $L$.  One has
\[
a_p(f)
=
\left(\tfrac{k_d}{p}\right)H_{p^2}(\mathrm{HD}_2)
=
\left(\tfrac{-1}{p}\right)
\bigl(\psi_E(\mathfrak B)+\overline{\psi_E(\mathfrak B)}\bigr)=H_p(\mathrm{HD}_3)-\left(\tfrac{-1}p\right)p,
\]
since $\psi_E(\mathfrak B)= p$ when $\left(\tfrac{-2}p\right)=-1$. By explicit computation of Fourier coefficients, the corresponding modular form is
identified as \(f=f_{32.3.d.a}\).

\end{example}

\begin{example}
For $d=4$, consider the elliptic curve
\[
E:\ y^2=x\Bigl(x^2+x+\frac{z}{4}\Bigr),
\qquad
z=\frac12-\frac{\sqrt2}{3}. 
\]
Then $
t=4z(1-z)=\frac19.
$
The curve $E$ has CM by the maximal order of 
$
K=\Q(\sqrt{-6}),
$
and is defined over $F=\Q(\sqrt2)$. 
The isomorphism class of $E$ labeled in LMFDB is \href{https://www.lmfdb.org/EllipticCurve/2.2.8.1/81.1/b/3}{$2.2.8.1-81.1-b3$}. 

The associated hypergeometric data are
\[
\mathrm{HD}_{2}:=\left\{ \begin{array}{cc} \frac14, &\frac 34\\ 1,&1\end{array};  \frac12-\frac{\sqrt2}{3} \right\},
\qquad
\mathrm{HD}_{3}:=\left\{ \begin{array}{ccc} \frac12, &\frac 14, &\frac 34\\ 1,&1, &1\end{array};  \frac 19 \right\}.
\]
Again viewing $\psi_E$ as a Hecke character on the ideals of $K$, for primes $p$ with
$\bigl(\tfrac{-6}{p}\bigr)=1$, we conclude that the
$p$-th Fourier coefficient of the associated weight-$3$ modular form satisfies
$$
a_p=H_{p}(\mathrm{HD}_{3})
     \;-\; \left(\tfrac{2}{p}\right)p=
     \begin{cases}
&\psi_E(\mathfrak B \mathfrak B^\sigma)+ \overline{ \psi_E(\mathfrak B \mathfrak B^\sigma)}=\left(\tfrac{-2}p\right)\left(\psi_E(\mathfrak p)+\overline{\psi_E(\mathfrak p)}\right),
\quad \mbox{if }  \bigl(\tfrac{2}{p}\bigr)=1,\\
&\left(\tfrac{-2}p\right)\left(\psi_E(\mathfrak p)+\overline{\psi_E(\mathfrak p)}\right),
\quad \mbox{if }  \bigl(\tfrac{2}{p}\bigr)=-1. 
\end{cases}
$$
where $\wp$ is a prime ideal of $K$ lying above $p$,   $\wp =\mathfrak B \mathfrak B^\sigma$  in $L$ if $\bigl(\tfrac{2}{p}\bigr)=1$, and $\sigma \in \Gal(L/\Q)$ fixes $\sqrt{-6}$ and $\sqrt{2}\mapsto -\sqrt 2$. 
We then identify
\(f=f_{24.3.h.a}\).

\end{example}

\bigskip

\section{Proofs}\label{sec:proof}

We begin with Fourier expansions of the CM forms in \zcref{tab:main} and then prove the corresponding $L$-value formulas.  The remaining cases in \zcref{append:all} follow by similar arguments.  

\bigskip
Let
\begin{equation}
    \tau_D=\begin{cases}
    \sqrt{-D/4}, &\text{for } D= 8,12, 16, \\
        \frac{1+\sqrt{-D}}{2}& \text{for } D= 7, 11, 19, 27, 43, 67, 163.
    \end{cases}
\end{equation} 
For each $D$, we express the CM cusp form of weight $3$ of level $D$ as 
$$
   f_D(q) = \frac 12\sum_{m,n\in \Z}(m+n\ol{\tau_D})^2q^{|m+n\tau_D|^2}.
$$
For $D=15$ and $D=24$,  we have two weight $3$ cusp forms that have CM by $\Q(\sqrt{-D})$ with level $D$  for each $D$. Namely,
$$
f_{24, \pm}(q) = \frac 12\sum_{m,n\in \Z}(m+n\sqrt{-6})^2q^{m^2+6n^2} \pm \frac 12\sum_{m,n\in \Z}(m\sqrt{2}+n\sqrt{-3})^2q^{2m^2+3n^2},  
$$
$$
f_{15, \pm}(q)= \frac 12 \sum_{m,n\in \Z}\left(m+n\frac{1-\sqrt{-15}}2\right)^2q^{\left|m+n\frac{1+\sqrt{-15}}2\right|^2}\pm \frac12\sum_{i=1}^2\sum_{m,n\in \Z}\alpha_i(m,n)^2q^{|\alpha_i(m,n)|^2},  
$$
where 
$$
\alpha_1(m,n)=m\sqrt 5 +n\sqrt{-3}, \quad  \alpha_2(m,n)=\frac{(2m+1)(\sqrt 5 -\sqrt{-3})}2-n\sqrt{-3}.
$$

The twisted cusp forms
$f_D\otimes \chi_{-4}$, for $D= 8, 12,7$, yield the following explicit $q$-expansions respectively. 
\begin{align*}
    f_{32}(q)\coloneqq f_{8}\otimes \chi_{-4}(q)=& \frac 12\sum_{m,n\in \Z}(m+n\sqrt{-2})^2q^{m^2+2n^2} -\frac 12\sum_{m,n\in \Z}(2m+n\sqrt{-2})^2q^{4m^2+2n^2}\\
     &-\sum_{m,n\in \Z}(2m+1+(2n+1)\sqrt{-2})^2q^{(2m+1)^2+2(2n+1)^2},\\
    f_{48}(q)\coloneqq f_{12}\otimes \chi_{-4}(q)=&\frac 12\sum_{m,n\in \Z}(m+n\sqrt{-3})^2q^{m^2+3n^2} -\sum_{m,n\in \Z}(2m+n\sqrt{-3})^2q^{4m^2+3n^2}\\
       &+\sum_{m,n\in \Z}(2m+2n\sqrt{-3})^2q^{4m^2+12n^2},\\
    f_{112}(q)\coloneqq f_{7}\otimes \chi_{-4}(q)=&\frac 12\sum_{m,n\in \Z}(m+2n\sqrt{-7})^2q^{m^2+28n^2} -\frac12\sum_{m,n\in \Z}(2m+n\sqrt{-7})^2q^{4m^2+7n^2}.
\end{align*}
Similarly, the twists $f_{16}\otimes \chi_{-8}$  and $ f_{16}\otimes \chi_{-3}$ give the following $q$-expansions for the modular forms $f_{64}$ of level $64$ and $f_{144}$ of level $144$, respectively:

$$
  f_{64}(q)= -\frac 12\sum_{m,n\in \Z}(m+2ni)^2q^{m^2+4n^2} + \sum_{m,n\in \Z}(2m+1-4ni)^2q^{(2m+1)^2+16n^2},
$$
and 
\begin{align*}
    f_{144}(q)=&-\frac12\sum_{m,n\in \Z}(2m+1+2ni)^2q^{|2m+1+2ni|^2}+\sum_{m,n\in \Z}(2m+1+6ni)^2q^{|2m+1+6ni|^2}\\
     &+\sum_{m,n\in \Z}(6m+3+2ni)^2q^{|6m+3+2ni|^2} -\frac 32\sum_{m,n\in \Z}(6m+3+3ni)^2q^{|6m+3+3ni|^2}.
\end{align*}

\subsection{The \texorpdfstring{$L$}{L}-values}

We include the following lemma for the sake of completeness, even though it may already be known.

\begin{lemma} \label{lemma:G2vaule}
 Set $\tau_0 \in \H$. If $f$ is a cusp form having Fourier expansion 
    $$
   f(q) = \frac 12\sum_{m,n \in \Z}(m+n\ol{\tau_0})^2q^{|m+n\tau_0|^2}.
$$
Then 
$$
  L(f,2)= G_2^\ast(\tau_0).
$$
\end{lemma}
\begin{proof}
By definition
\begin{align*}
    L(f,s) &=\frac{1}{2}\sum_{m,n\in \Z}' \frac{ (m+n\ol{\tau_0})^2 }{|m+n\tau_0|^{2s}}=\frac{1}{2}\sum_{m,n\in \Z}' \frac{1}{(m+n\tau_0)^{2}|m+n\tau_0|^{2s-4} },
\end{align*}
where $'$ denotes excluding the pair $(0,0)$ from the sum.
Taking $s \rightarrow 2$, we have 
$L(f,2)=G_2^*(\tau_0)$, 
which completes the proof. 
\end{proof}

Before we proceed to prove our main results, we use the following known results (see \cite{EHMM2, LLT, Osburn-Straub} for example) to illustrate our ideas through moduli interpretation, Ramanujan alternative bases, and the hypergeometric transformations. 
\begin{theorem}
For $D=12$, we have 
    \begin{align*}
          L(f_{12},2) =& \frac{\pi^2}{12}\pFq32{\frac 12&\frac 14&\frac 34}{&1&1}{-\frac{16}9}= \frac{\pi^2}{18}\pFq32{\frac 12&\frac 13&\frac 23}{&1&1}1=\frac{\pi^2}{6\sqrt 5}\pFq32{\frac 12&\frac 16&\frac 56}{&1&1}{\frac4{125}};
    \end{align*}
for $D=8$, 
    \begin{align*}
           L(f_{8},2) =& \frac{\pi^2}{24}\pFq32{\frac 12&\frac 14&\frac 34}{&1&1}{1}= \frac{3\pi^2}{56}\pFq32{\frac 12&\frac 14&\frac 34}{&1&1}{\frac 1{2401}} \\
           =&  \frac{\pi^2}{12\sqrt 2}\pFq32{\frac 12&\frac 12&\frac 12}{&1&1}{-1} =  \frac{\pi^2}{6\sqrt{10}}\pFq32{\frac 12&\frac 16&\frac 56}{&1&1}{\frac{27}{125}}. 
    \end{align*}
\end{theorem}
\begin{proof}
We first consider the case of $D=12$. Recall that  $G_2^\ast(\frac{1+\tau_D}2)=0$ and $t_4(\frac{1+\tau_D}2)=-16/9$.  
Thus, by \zcref{lemma:G2vaule} and \zcref{tab:haupymodul}, 
 \begin{align*}
L(f_{12},2)=& G_2^\ast(\tau_D) = G_2^\ast(\tau_D+1)=  G_2^\ast(\tau_D+1)-\frac12 G_2^\ast\left(\frac{1+\tau_D}2\right) =\frac12G_{2,2}\left(\frac{1+\tau_D}2\right)  \\
=&\frac{\pi^2}{12}\pFq32{\frac 12&\frac 14&\frac 34}{&1&1}{t_4\left(\frac{1+\tau_D}2\right)}
=\frac{\pi^2}{12}\pFq32{\frac 12&\frac 14&\frac 34}{&1&1}{-\frac{16}9}.
    \end{align*}
Noting $-1/\tau_D=\tau_D/3$ and $t_3(-1/\tau_D)=1$, 
we then obtain 
\begin{align*}
2L(f_{12},2)=&2G_2^\ast(\tau_D)
=G_2^\ast(\tau_D)-\frac13 G_2^\ast(\tau_D/3) +\frac13 G_2^\ast(-1/\tau_D) +G_2^\ast(\tau_D)\\
=&\frac13G_{2,3} (-1/\tau_D) -G_2^\ast(\tau_D)+G_2^\ast(\tau_D)= \frac{\pi^2}{9}\pFq32{\frac 12&\frac 13&\frac 23}{&1&1}1.
    \end{align*}

Since the point $\frac{-1}{\tau_D}$ lives in the desired fundamental domains for $\G_0^+(3)$ but not in $\G_0(1)$ (\zcref{fig:fundamentaldomains}), to derive the third identity,  we apply the inversion $\tau\mapsto -1/\tau$ on the modular form $E_4(\tau)$ as follows. 
\begin{align}\label{eqn:E4trans}
\pFq{3}{2}{\frac{1}{2}&\frac{1}{6}&\frac{5}{6}}{&1&1}{\frac{1728}{j(-1/\tau)}}^2 = \frac1{\tau^4}\cdot \pFq{3}{2}{\frac{1}{2}&\frac{1}{6}&\frac{5}{6}}{&1&1}{\frac{1728}{j(\tau)}}^2. 
\end{align}
Thus, we have 
$$
3\cdot\pFq32{\frac 12&\frac 16&\frac{5}{6}}{&1&1}{\frac{4}{125}}=  \sqrt 5\pFq{3}{2}{\frac{1}{2}&\frac{1}{3}&
\frac{2}{3}}{&1&1}{1} 
$$ 
after applying continuation to  \zcref{prop:hgtransformations} with $z=1/2$ \footnote{Similarly, we can obtain that 
$$
\frac 32\cdot \pFq32{\frac 12&\frac 14&\frac34}{&1&1}{-\frac{16}9}=  \pFq{3}{2}{\frac{1}{2}&\frac{1}{3}&\frac{2}{3}}{&1&1}{1}.
$$}.

\medskip
For the case $D=8$, the result
$$
 L(f_{8},2) = \frac 14G_{2,2}(\tau_D/2)=\frac{\pi^2}{24}\pFq32{\frac 12&\frac 14&\frac 34}{&1&1}{1}
$$
can be obtained similarly by using $-1/\tau_D=\tau_D/2$ and  $t_4(-1/\tau_D)=1$. 
Applying the identities in \zcref{prop: gamma06} and \zcref{prop:hgtransformations}, we have that 
\begin{align*}
   G_{2,2}(\tau_D/2)=\frac 97G_{2,2}(3\tau_D/2)=\frac{3\pi^2}{56}\pFq32{\frac 12&\frac 14&\frac 34}{&1&1}{\frac1{2401}}
\end{align*}
and 
$$
\pFq{3}{2}{\frac{1}{2}&\frac{1}{2}&
\frac{1}{2}}{&1&1}{-1} =\frac1{\sqrt 2}\pFq{3}{2}{\frac{1}{2}&\frac{1}{4}&\frac{3}{4}}{&1&1}{1}. 
$$

To derive the final identity, we notice that $-1/\tau_D$ is not in the given fundamental domain of $\G_0(1)$. Combining with $z=-1$ in \zcref{prop:hgtransformations} and \eqref{eqn:E4trans}, we deduce that 
$$
\pFq{3}{2}{\frac{1}{2}&\frac{1}{6}&\frac{5}{6}}{&1&1}{\frac{1728}{j(-1/\tau_D)}}^2 =  \frac 58\pFq{3}{2}{\frac{1}{2}&\frac{1}{4}&\frac{3}{4}}{&1&1}{1}^2.
$$
Since the hypergeometric evaluation is positive, we get the desired identity.  
\end{proof}

For the case of $D=7$ and $D=16$, the identities can be derived in the same manner and also can be viewed as  corollaries of the known results in \cite[et al.]{LLT, Ono-book}:
\begin{align*}
     L(f_{7},2) =\frac{\pi^2}{8\sqrt 7}\pFq32{\frac 12&\frac 12&\frac 12}{&1&1}{\frac 1{64}}, \quad  L(f_{16},2) = \frac{\pi^2}{16}\pFq32{\frac 12&\frac 12&\frac 12}{&1&1}{1}. 
\end{align*}
\begin{corollary}\label{cor:716}
    We have  
    \begin{align*}
           L(f_{7},2) =& \frac{\pi^2}{21}\pFq32{\frac 12&\frac 14&\frac 34}{&1&1}{-\frac{256}{3969}}  = \frac{\pi^2}{8\sqrt 7}\pFq32{\frac 12&\frac 12&\frac 12}{&1&1}{\frac 1{64}} \\
           =&  \frac{\pi^2}{2\sqrt{105}}\pFq32{\frac 12&\frac 16&\frac 56}{&1&1}{-\frac{64}{125}}= \frac{2\pi^2}{\sqrt{1785}} \cdot \pFq32{\frac 12&\frac16&\frac56}{&1&1}{\frac{64}{614125}}.  \\
   L(f_{16},2) =& \frac{\pi^2}{12}\pFq32{\frac 12&\frac 14&\frac 34}{&1&1}{\frac{32}{81}} =  \frac{\pi^2}{16}\pFq32{\frac 12&\frac 12&\frac 12}{&1&1}{1}  =  \frac{\pi^2}{8}\pFq32{\frac 12&\frac 12&\frac 12}{&1&1}{-8}  \\
   =& \frac{\pi^2}{8\sqrt{3}}\pFq32{\frac 12&\frac 16&\frac 56}{&1&1}{1} = \frac{\pi^2}{2\sqrt{33}}\pFq32{\frac 12&\frac 16&\frac 56}{&1&1}{\frac 8{1331}}. 
    \end{align*}
\end{corollary}

\medskip
Applying similar ideas, we can obtain the following $\C$-version of the modularity. 
\begin{theorem}
For $D=24$, we have 
    \begin{align*}
          L(f_{24,+},2) = \frac{\pi^2}{9}\pFq32{\frac 12&\frac 13&\frac 23}{&1&1}{\frac 12}, \quad  L(f_{24,-},2) =& \frac{\pi^2}{12}\pFq32{\frac 12&\frac 14&\frac 34}{&1&1}{\frac 19}; 
    \end{align*}
for $D=15$, 
    $$  L(f_{15,+},2) = \frac{4\pi^2}{45}\pFq32{\frac 12&\frac 13&\frac 23}{&1&1}{\frac 4{125}}.$$
\end{theorem}
\begin{proof}
Let $s=\sqrt{-6}$. Applying \zcref{lemma:G2vaule}, we get 
 \begin{align*}
L(f_{24,+},2)=& G_2^\ast(s) -\frac 13G_2^\ast\left(-\frac{\sqrt 2}{\sqrt{-3}}\right)= G_2^\ast(s) -\frac 13G_2^\ast(s/3) \\
=&  \frac 13G_{2,3}^\ast(s/3)=\frac{\pi^2}{9}\pFq32{\frac 12&\frac 13&\frac 23}{&1&1}{\frac 12}.
    \end{align*}
. 

Similarly, 
 \begin{align*}
L(f_{24,-},2)=& G_2^\ast(s) -\frac 12G_2^\ast\left(\frac{\sqrt{-3}}{\sqrt 2}\right)= \frac 12G_{2,2}(s/2)=  \frac 12G_{2,2}(-1/s)
= \frac{\pi^2}{12}\pFq32{\frac 12&\frac 14&\frac 34}{&1&1}{\frac 19}.
    \end{align*}
on. 

We now consider the $L$-value of $f_{15,+}$. Setting $s=\sqrt{-15}$, we have 
 \begin{align*}
L(f_{15,+},2)=& G_2^\ast\left(\frac{1+s}2\right) -\frac 13G_2^\ast\left(\frac s3\right)-\frac 13 \mathcal G_{(1,2)}^\ast\left(\frac{-3+s}6\right)\\
=& G_2^\ast\left(\frac{1+s}2\right) -\frac 13G_2^\ast\left(\frac s3\right)-\frac 13\left(G_2^\ast\left(\frac{-3+s}6\right)-G_2^\ast\left(\frac s3\right)\right)\\
=& \frac13G_{2,3}\left(\frac{-3+s}6\right)  = \frac 13 \cdot \frac 45 G_{2,3}\left(\frac s3\right)
    \end{align*}
by \zcref{prop: gamma06}. Hence, we conclude that
$$
  L(f_{15,+},2)
  = \frac 13 \cdot \frac 45 \cdot \frac{\pi^2}{3}  \pFq32{\frac 12&\frac 13&\frac 23}{&1&1}{\frac{4}{125}}.
$$
\end{proof}

\begin{remark}
Note that $$
 L(f_{24,-},2)= L(f_{24,+}\otimes \chi_8,2) = \sqrt 2L(f_{24,+},2),$$
 which can also be seen from \zcref{prop:hgtransformations}.
For the form $f_{15,-}$, one has 
  \begin{align*}
     L(f_{15,-},2)=& L(f_{15,+}\otimes \chi_5,2)= \frac{8\sqrt 5}{225}\pi^2\pFq32{\frac 12&\frac 13&\frac 23}{&1&1}{\frac 4{125}}\\
     =&\frac 15 G_{2,5}\left(\frac{5+\sqrt{-15}}{10}\right).  
      \end{align*}
\end{remark}

\medskip

In the end of this section, we relate the $L$-values listed in  \zcref{tab:main} with the CM-values of the targeted Eisenstein series, which will imply the results. 

\begin{theorem}\label{thm: special Lvaule} The $L$-values of the CM forms in our table are as follows. 
    \begin{align*}
&L(f_{27},2)=\frac13G_{2,3}(\zeta_3),\\
&L(f_{32},2)=\frac {\pi^2}8\theta_4\left(\sqrt{2}i\right)^4,\\ 
&L(f_{48},2)=\frac {\pi^2}8\theta_4\left(2\zeta_3\right)^4,\\ 
&L(f_{64},2)=\frac {\pi^2}8\theta_4\left(2i\right)^4,\\
&L(f_{112},2)=\frac 14G_{2,4}\left(\sqrt{7}i/2 \right)=\frac {\pi^2}8\theta_4\left(1+\sqrt{7}i\right)^4, \\ 
&L(f_{144},2)=-\frac23G_{2,2}\left(i\right)+G_{2,2}\left(3i\right)+\frac 19G_{2,2}\left(i/3\right)=\frac4{3\sqrt 3}G_{2,2}(i). 
\end{align*} 
\end{theorem}
In the proof, we will handle the first five cases, for which we can conclude the identity by straightforward calculation. In the case of $144$, we  apply the relations in \zcref{prop:curlyg2asttoclassicales, prop:G2astfricke} with the modularity of $G^*_2$.   
\begin{proof}
The value $L(f_{27},2)$ is obtained immediately from the facts that 
$$
  f_{27}(q)=\frac12 \sum_{m,n}(m+3n \ol\zeta_3)^2q^{|m+3n \zeta_3|^2} 
$$
and 
$G_2^\ast(\zeta_3)=0$.   From the $q$-expansion of $f_{64}$,  one has 
$$
    L(f_{64},2)=- G_2^\ast(2i)+2G_2^\ast(4i)-\frac12 G_2^\ast(2i)=2G_2^\ast(4i)-\frac32 G_2^\ast(2i)=\frac{\pi^2}{8}\theta_4^4(2i). 
$$
The last equality is due to the fact that $G_2^\ast(i)=0$ and  \zcref{prop: g22halfintegershift}. Likewise,
$$
L(f_{112},2)=G_2^\ast(2\sqrt{7}i)-\frac14G_2^\ast\left(\frac{\sqrt{7}i}2\right) = \frac14G_{2,4}\left(\frac{\sqrt{7}i}2 \right). 
$$
Then  \zcref{prop: g22halfintegershift} and  $\theta_3(2\tau)^4=\theta_4(2\tau+1)^4$
give the desired identity.  

We now consider the cases  $f_{48}$ and $f_{32}$. By straightforward computations, we observe that 
\begin{align*}
L(f_{32},2)=&-G_2^\ast(\sqrt{-2})+\frac14 G_2^\ast\left(\frac{\sqrt{-2}}2\right)+2G_2^\ast(2\sqrt{-2})-\frac12G_2^\ast(\sqrt{-2})\\ 
 =&2G_2^\ast(2\sqrt{-2})-\frac32G_2^\ast(\sqrt{-2})+\frac14 G_2^\ast\left(\frac{\sqrt{-2}}2\right)\\ 
 =&\frac14\left(4G_{2,2}(\sqrt{-2})-G_{2,2}\left(\frac{\sqrt{-2}}2\right)\right)=\frac{\pi^2}6\cdot \frac34 \theta_4(\sqrt{-2})^4,
\end{align*}
and 
$$
   L(f_{48},2)=\frac 12 G_2^\ast(\sqrt{-3})-\frac12 G_{2,2}\left(\frac{\sqrt{-3}}2\right).
$$
To proceed, we apply the fact that $G_2^\ast(\zeta_3)=0$ and the relation in \zcref{prop: g22halfintegershift}.   
This gives us 
\begin{align*}
L(f_{48},2)=&\frac 12  G_{2,2}\left(\frac{\sqrt{-3}}2\right)+\frac 14 G_{2,2}\left(\zeta_3\right)\\
  =& \frac 14\left(2G_{2,4}\left(\zeta_3\right)-4G_{2,2}\left(\zeta_3\right)+G_{2,2}\left(\zeta_3\right)\right)=\frac 14\left(2G_{2,4}\left(\zeta_3\right)-3G_{2,2}\left(\zeta_3\right)\right),
\end{align*}
and leads to the identity as stated.

To compute $ L(f_{144},2)$, we first observe that 
\begin{align*}
    L(f_{144},2)=&\frac{1}{4}\mathcal{G}^*_{2,(1;2)}\left(\frac{i}{2}\right)-\frac1{18}\mathcal{G}^*_{2,(1;2)}\left(\frac{i}{6}\right)-\frac12\mathcal{G}^*_{2,(1;2)}\left(\frac{3i}{2}\right)+\frac{1}{3}\mathcal{G}^*_{2,(1;2)}\left(i\right). 
\end{align*}
Using \zcref{prop:curlyg2asttoclassicales,  prop:G2astfricke},
\begin{align*}
    L(f_{144},2)=&-\frac14G_{2,2}\left(\frac{i}2\right)+\frac1{18}G_{2,2}\left(\frac{i}{6}\right)+\frac 12G_{2,2}\left(\frac{3i}{2}\right)-\frac 13G_{2,2}\left(i\right)\\
    &+\frac 14G_2^*\left(i\right)-\frac 1{18}G^*_2\left(\frac{i}{6}\right)-\frac 12G^*_2\left(3i\right)+\frac 13G^*_2\left(2i\right)\\
    =&-\frac12G_{2,2}\left(i\right)+G_{2,2}\left(3i\right)+\frac 19G_{2,2}\left(\frac{i}{3}\right)-\frac 13G_{2,2}\left(i\right)\\
    &+\frac 9{18}G^*_2\left(3i\right)-\frac 12G^*_2\left(3i\right)-\frac 16G_{2,2}\left(i\right)\\
    =&-\frac23G_{2,2}\left(i\right)+G_{2,2}\left(3i\right)+\frac 19G_{2,2}\left(\frac{i}{3}\right).
\end{align*}

Finally, we apply  \zcref{prop: gamma06} to connect the values $G_{2,2}\left(3i\right)$ and $G_{2,2}\left(i/3\right)$ with $G_{2,2}\left(i\right)$ by 
$$
G_{2,2}(3i)=\frac{9+6\sqrt3+2\cdot 3^{3/4}\sqrt 2}{27}G_{2,2}(i), \quad G_{2,2}(i/3)=\frac{9+6\sqrt3-2\cdot 3^{3/4}\sqrt 2}{3}G_{2,2}(i).
$$
Therefore,  we obtain
$$ 
  L(f_{144},2)=\frac4{3\sqrt 3}G_{2,2}(i),
$$
after simplification. 
\end{proof}
Combining \zcref{thm: special Lvaule}, \zcref{tab:haupymodul}, and also \zcref{prop:hgtransformations}, we have 
$$
L(f_{144},2)=\frac {2\pi^2}{9\sqrt 3}\pFq32{\frac 12&\frac14&\frac34}{&1&1}{\frac{32}{81}}=\frac {\pi^2}{9}\pFq32{\frac 12&\frac16&\frac56}{&1&1}{1},
$$
and 
$L(f_{144},2)=\frac{8\sqrt 3}9L(f_{16},2)$, by \zcref{cor:716}. 

\medskip

\subsection{\texorpdfstring{$L(f_{300.3.g.b},2)$}{L(f 300.3.g.b,2)}} In this section, 
we confirm the
observation in \cite[Observation $5.2$]{SRiM} through the moduli interpretation and modularity.  
\medskip

\begin{theorem}\label{thm: 300Lvalues}
    $$
      L(f_{300},2)= \frac 4{25} \pi^2E_4(\sqrt{-3})^{1/2}=\frac 4{25} \pi^2\pFq32{\frac12&\frac16&\frac56}{&1&1}{\frac 4{125}}.
    $$    
\end{theorem}

To proceed, we prove some useful lemmas. 
\begin{lemma}\label{lem: level5hauptmoduln}
Let  $x(\tau)=t_5(\tau)\coloneqq\eta(5\tau)^6/\eta(\tau)^6$ and $y(\tau)=t_{25}(\tau)\coloneqq\frac{\eta(25\tau)}{\eta(\tau)}$ be hauptmoduln for $\G_0(5)$ and $\G_0(25)$, respectively. Then
$$y(\tau)^6=x(\tau)x(5\tau),$$
\begin{align*}
x=&25y^5 + 25y^4 + 15y^3 + 5y^2 + y\\
xj=&30517578125x^6 + 7324218750x^5 + 615234375x^4 + 20312500x^3 + 196875x^2 + 750x + 1,\\
j(5\tau)=& \frac{(5y^2 + 5y + 1)^3(25y^4 + 25y^3 + 20y^2 + 5y + 1)^3(25y^4 + 5y^2 + 1)^3}{y^5(25y^4 + 25y^3 + 15y^2 + 5y + 1)^5}\left(\tau\right).
\end{align*}
Moreover,  when $s=\sqrt{-3}$, 
the value $y(s/5)$  is the real root of the polynomial $$50z^3  - (25-15\sqrt 5)z^2 -(5-5\sqrt5)z+ 4\sqrt 5- 10.$$  
\end{lemma}

\begin{proof} 
By the relations between the modular functions $x(5\tau)$ and $j(5\tau)$, and $x(5\tau)$ and $y(\tau)$, we obtain that 
$$
j(5\tau)= \frac{(5y^2 + 5y + 1)^3(25y^4 + 25y^3 + 20y^2 + 5y + 1)^3(25y^4 + 5y^2 + 1)^3}{y^5(25y^4 + 25y^3 + 15y^2 + 5y + 1)^5}\left(\tau\right).
$$
As a consequence, we use the special value $j(s)=54000$ to obtain the special value $y(s/5)$ via the relationship. By approximating the value $y(s/5)$ using the $q$-expansion of $y(\tau)$, one   deduces that $y(s/5)$ is the real root with absolute less $1$ of the polynomial  
\begin{align*}
 &125y^6 - 125y^5 - 50y^4 - 75y^3 - 10y^2 - 5y + 1\\
 =&\frac1{20}(50y^3  - (25-15\sqrt 5)y^2 -(5-5\sqrt5)y+ 4\sqrt 5- 10)\\
  &\times (50y^3  - (25+15\sqrt 5)y^2 -(5+5\sqrt5)y- 4\sqrt 5- 10).
\end{align*}
Hence $y(s/5)$  is the real root of $50y^3  - (25-15\sqrt 5)y^2 -(5-5\sqrt5)y+ 4\sqrt 5- 10$. 
\end{proof}

\medskip

\begin{lemma} \label{lem: level5forms}
Using the same notations as in \zcref{lem: level5hauptmoduln},
\begin{align*}
   G_2^\ast\otimes\chi_5(\tau)=& \frac{6y(5y^2-1)}{125y^4+100y^3+45y^2+10y+1}G_{2,5}(\tau)\\
   G_{2,5}(5\tau)=&\frac{5y^4+10y^3+9y^2+4y+1}{125y^4+100y^3+45y^2+10y+1}G_{2,5}(\tau),\\
  G_{2,5}(\tau)^2 =&40\cdot \frac{125x^2 + 22x + 1}{5x^2 + 10x + 1}(\tau) \cdot G_4(5\tau).
\end{align*}
\end{lemma}
The main idea is that  $ g=G_2\otimes\chi_5(\tau)$ and $G_{2,5}(5\tau)$ are modular forms on $\G_0(25)$, and hence $g/G_{2,5}$ is a modular function on $\G_0(25)$ that can be written as a rational function of $y$. Similarly, both $ G_{2,5}(\tau)^2$ and $G_4(5\tau)$ are  weight 4 form on $\G_0(5)$. Therefore, their ratio is a rational function of $x(\tau)$. We will omit the details here.

\bigskip

Let $f_{300}$ be the cusp form $f_{300.3.g.b}$ in LMFDB label. Its $q$-expansion is 
\begin{align*}
f_{300}(q)=&\sum_{j=1}^4\left(\frac{j}{5}\right)\sum'_{m,n \in \Z}(5m+(5n+1)(j+\sqrt{-3}))^2q^{|5m+(5n+1)(j+\sqrt{-3})|^2}\\
&+\sum_{j=1}^4\left(\frac{j}{5}\right)\sum'_{m,n \in \Z}(5m+(5n+2)(j+\sqrt{-3}))^2q^{|5m+(5n+2)(j+\sqrt{-3})|^2}\\
&-\frac{1}{2}\sum'_{m,n \in \Z}(5m+n\sqrt{-3})^2q^{|5m+n\sqrt{-3}|^2}+\frac{1}{2}\sum'_{m,n \in \Z}(m+5n\sqrt{-3})^2q^{|m+5n\sqrt{-3}|^2}. 
\end{align*}

\begin{lemma} \label{lem: 300Lvalue} 
We have
\begin{align*}
  L(f_{300},2)=&\frac{\sqrt 5}{25}  G_2^*\otimes\chi_5\left(\frac{\sqrt{-3}}5\right)+\frac1{25}G_{2,5}\left(\frac {\sqrt{-3}}5\right) + \frac 15G_{2,5}\left(\sqrt{-3}\right). 
\end{align*}

\end{lemma}
\begin{proof}

According to the $q$-expansions and the fact that 
$$
  \mathcal{G}^*_{2,(1;5)}(\tau)+\mathcal{G}^*_{2,(2;5)}(\tau) = \frac12\left(G_2^*(\tau)-G_2^*(5\tau)\right),
$$
 we can conclude that  
\begin{align*}
    L(f_{300.3.g.b},2)=&\frac{1}{25}\sum_{j=1}^4\left(\frac{j}{5}\right)\left(G_2^*\left(\frac{j+\sqrt{-3}}{5}\right)-G_2^*(\sqrt{-3})\right)-\frac{1}{25}G_2^*\left(\frac{\sqrt{-3}}{5}\right)+G_2^*(5\sqrt{-3}).
\end{align*}
Moreover, the first term can be simplified as follows. 
\begin{align*}
\sum_{j=1}^4&\left(\frac{j}{5}\right)\left(G_2^*\left(\frac{j+\sqrt{-3}}{5}\right)-G_2^*(\sqrt{-3})\right)=\sum_{j=1}^4\left(\frac{j}{5}\right)G_2^*\left(\frac j 5+ \frac{\sqrt{-3}}{5}\right)
=\sqrt 5 \cdot \left(G^\ast_2\otimes\chi_5\right)(\sqrt{-3}/5). 
\end{align*}
This completes the proof. 
\end{proof}

\begin{proof}[Proof of \zcref{thm: 300Lvalues}]

Combing the relations given in Lemmas \ref{lem: 300Lvalue}, \ref{lem: level5forms}, and \ref{lem: level5hauptmoduln}, we deduce that 
\begin{align*}
     L(f_{300},2)=&\frac{ 150y^4 + (150+30\sqrt 5)y^3 + 90y^2 + (30-6\sqrt 5)y + 6}{3125y^4 + 2500y^3 + 1125y^2 + 250y + 25}G_{2,5}\left(\frac {\sqrt{-3}}5\right)\\
         =& \frac4{25}\cdot \pi^2 E_4^{1/2}(\sqrt{-3}),
\end{align*}
by taking 
$$
G_{2,5}(\sqrt{-3}/5) = 2\sqrt 5\cdot \sqrt{\frac{125x^2 + 22x + 1}{5x^2 + 10x + 1}(\sqrt{-3}/5)} \cdot G_4(\sqrt{-3})^{1/2},
$$
because the ratio $G_{2,5}(\sqrt{-3}/5)/G_4(\sqrt{-3})^{1/2}$ is positive. 
\end{proof}

\medskip

\begin{example}
    As mentioned in \zcref{rmk:final}, one can compute $L$-values of certain CM forms of higher weight in a similar way. Here we give two examples arising from powers of  Hecke characters in this paper.  For instance,
 \begin{equation}
        L(f_{8.5.d.a} ,4)= G_4(\sqrt{2}i) =\frac{\pi^4}{90} \pFq 32{\frac 12&\frac 16&\frac 56}{&1&1}{\frac{1728}{8000}}^2,
    \end{equation}
    where  $
        f_{8.5.d.a} = \frac 12\sum_{m,n\in \Z}(m-n\sqrt{-2})^4q^{m^2+2n^2}$. For \begin{align}
        f_{32.7.d.a}=& \frac 12\sum_{m,n\in \Z}(m-n\sqrt{-2})^6q^{m^2+2n^2}-\frac 12\sum_{m,n\in \Z}(2m-n\sqrt{-2})^6q^{4m^2+2n^2}\\
        &-\sum_{m,n\in \Z}(2m+1-(2n+1)\sqrt{-2})^6q^{(2m+1)^2+2(2n+1)^2},
    \end{align}
    we have
    \begin{align}
        L(f_{32.7.d.a} ,6)&= \frac{33}{32}G_6(\sqrt{2}i) -\frac1{64} (G_6(\sqrt{2}i/2)+2G_6((\sqrt{2}i+1)/2))=\frac{\pi^6}{768}\pfq{3}{2}{\frac{1}{2}&\frac{1}{2}&\frac{1}{2}}{&1&1}{-1}^3.  
    \end{align}
   When computing the Petersson inner product of a CM weight $2$ cuspform $h$  on $\Gamma_1(N)$ with real Fourier coefficients attached to a Hecke character $\psi_h$ of the CM field $K=\Q(\sqrt{d_K})$, 
    $$
      \gen{h,h}:=\int_{\G_1(N)\backslash \H}h\ol h dxdy, \quad \tau=x+iy,
    $$
    one has 
    $$
      \gen{h,h}= \frac{N}{4\pi^2}\frac{h_K}{w_K\sqrt{|d_K|}}L(\psi_h^2,2)\prod_{p\mid N}(1-\chi_K(p)p^{-1}). 
    $$
    For instance, as in \zcref{ex:EC32}, 
     $$
      \gen{f_{32.2.a.a},f_{32.2.a.a}}= L(f_{16.3.c.a},2)/\pi^2=\Omega_{-4}^2/32\pi^2.  
    $$
    Similarly, 
    $$
      \gen{f_{256.2.a.a},f_{256.2.a.a}}= \gen{f_{256.2.a.d},f_{256.2.a.d}}= 8\sqrt{2}L(f_{8.3.d.a},2)/\pi^2=\Omega_{-8}^2/6\pi^2.  
    $$
\end{example}

\medskip

\section{Appendix}\label{appendix}

\subsection{Fundamental Domains}
In \zcref{fig:fundamentaldomains}, we give the fundamental domains (the shaded regions) for the groups $\G_0(1)$, $\G_0(2)$, $\G_0(3)^+$,  and $\G_0(2)^+$, corresponding to Ramanujan's alternative bases given in \zcref{tab:haupymodul}. 

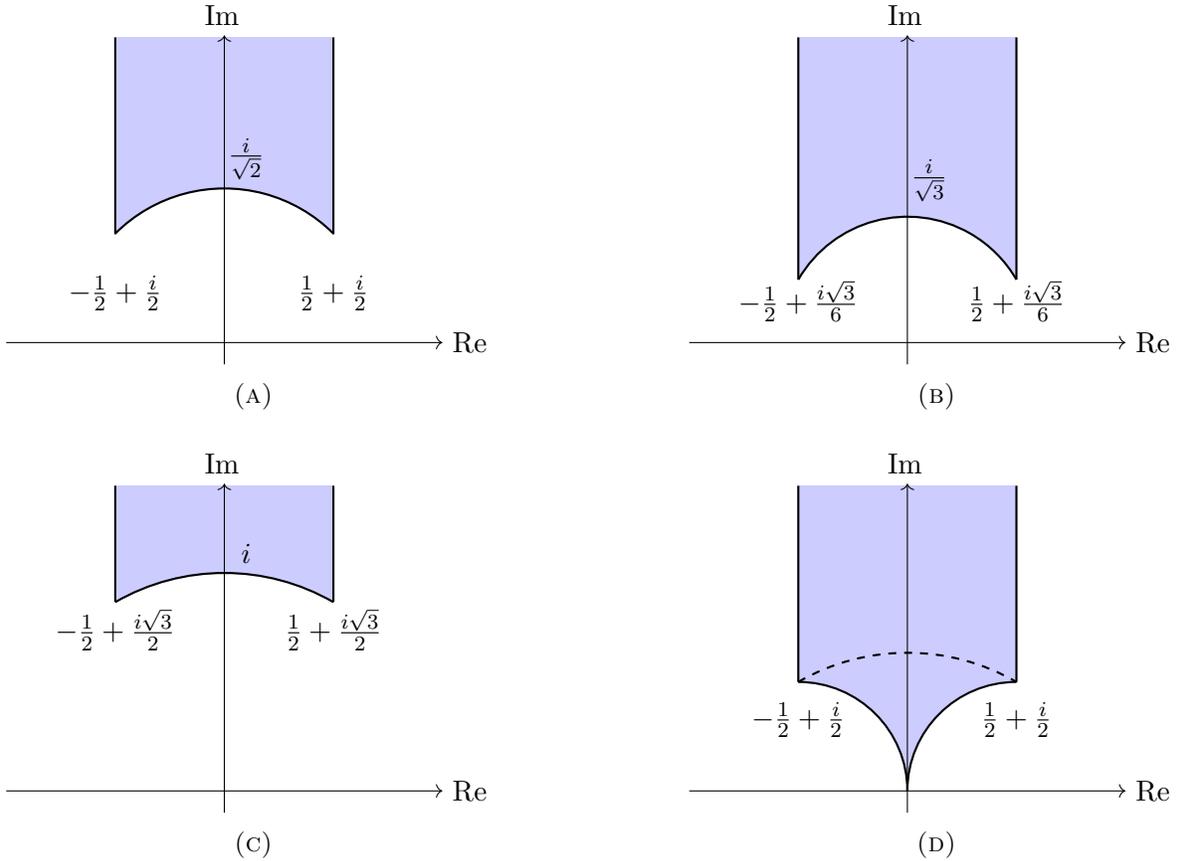
\begin{figure}[h!]
\centering

\begin{subfigure}[c]{0.45\textwidth}
\centering
\begin{tikzpicture}[scale=2.9]

\pgfmathsetmacro{\r}{1/sqrt(2)}
\pgfmathsetmacro{\h}{sqrt(\r*\r - 0.5*0.5)}
\pgfmathsetmacro{\ang}{acos(0.5/\r)}

\begin{scope}
  \clip (-0.5,0) rectangle (0.5,1.4);
  \fill[blue!20] (-1,0) rectangle (1,1.4);
  \fill[white] (0,0) circle (\r);
\end{scope}

\draw[thick] ({-0.5},{\h})
     arc[start angle={180-\ang}, end angle={\ang}, radius=\r];

\draw[thick] (-0.5,0.497) -- (-0.5,1.4);
\draw[thick] (0.5,0.497) -- (0.5,1.4);

\draw[->] (-1,0) -- (1,0) node[right]{$\Re$};
\draw[->] (0,-0.1) -- (0,1.41) node[above]{$\Im$};

\node[above] at (-0.5,0.1) {$-\frac{1}{2}+\frac{i}{2}$};

\node[above] at (0.5,0.1) {$\frac{1}{2}+\frac{i}{2}$};
\node[above] at (0.1,0.7) {$\frac{i}{\sqrt{2}}$};

\end{tikzpicture}

\caption{}
\end{subfigure}
\hfill
\begin{subfigure}[c]{0.45\textwidth}
\centering
\begin{tikzpicture}[scale=2.9]

\pgfmathsetmacro{\r}{1/sqrt(3)}
\pgfmathsetmacro{\h}{sqrt(\r*\r - 0.5*0.5)}
\pgfmathsetmacro{\ang}{acos(0.5/\r)}

\begin{scope}
  \clip (-0.5,0) rectangle (0.5,1.4);
  \fill[blue!20] (-1,0) rectangle (1,1.4);
  \fill[white] (0,0) circle (\r);
\end{scope}

\draw[thick] ({-0.5},{\h})
     arc[start angle={180-\ang}, end angle={\ang}, radius=\r];

\draw[thick] (-0.5,0.29) -- (-0.5,1.4);
\draw[thick] (0.5,0.29) -- (0.5,1.4);

\draw[->] (-1,0) -- (1,0) node[right]{$\Re$};
\draw[->] (0,-0.1) -- (0,1.41) node[above]{$\Im$};

\node[above] at (-0.5,0.05) {$-\frac{1}{2}+\frac{i\sqrt{3}}{6}$};

\node[above] at (0.5,0.05) {$\frac{1}{2}+\frac{i\sqrt{3}}{6}$};
\node[above] at (0.1,0.6) {$\frac{i}{\sqrt{3}}$};

\end{tikzpicture}
\caption{}
\end{subfigure}

\vspace{1em}

\begin{subfigure}[c]{0.45\textwidth}
\centering
\begin{tikzpicture}[scale=2.9]

\pgfmathsetmacro{\r}{1}
\pgfmathsetmacro{\h}{sqrt(\r*\r - 0.5*0.5)}
\pgfmathsetmacro{\ang}{acos(0.5/\r)}

\begin{scope}
  \clip (-0.5,0) rectangle (0.5,1.4);
  \fill[blue!20] (-1,0) rectangle (1,1.4);
  \fill[white] (0,0) circle (\r);
\end{scope}

\draw[thick] ({-0.5},{\h})
     arc[start angle={180-\ang}, end angle={\ang}, radius=\r];

\draw[thick] (-0.5,0.867) -- (-0.5,1.4);
\draw[thick] (0.5,0.867) -- (0.5,1.4);

\draw[->] (-1,0) -- (1,0) node[right]{$\Re$};
\draw[->] (0,-0.1) -- (0,1.41) node[above]{$\Im$};

\node[above] at (-0.5,0.6) {$-\frac{1}{2}+\frac{i\sqrt{3}}{2}$};
\node[above] at (0.5,0.6) {$\frac{1}{2}+\frac{i\sqrt{3}}{2}$};
\node[above] at (0.1,1) {$i$};

\end{tikzpicture}
\caption{}
\end{subfigure}
\hfill
\begin{subfigure}[c]{0.45\textwidth}
\centering
\begin{tikzpicture}[scale=2.9,baseline={(0,0)}]

\pgfmathsetmacro{\r}{0.5} 
\pgfmathsetmacro{\hL}{sqrt(\r*\r - (-0.5 + 0.5)^2)} 
\pgfmathsetmacro{\hR}{sqrt(\r*\r - (0.5 - 0.5)^2)} 

\pgfmathsetmacro{\angL}{acos((0.5-0.5)/\r)}   
\pgfmathsetmacro{\angR}{acos((-0.5+0.5)/\r)}  

\begin{scope}

    \clip (-0.5,0) rectangle (0.5,1.4);
   
    \fill[blue!20] (-1,0) rectangle (1,1.4);
   
    \fill[white] (-0.5,0) circle (\r);
  
    \fill[white] (0.5,0) circle (\r);
\end{scope}

\draw[thick] (-0.5,0.5) arc[start angle=90,end angle=0,radius=\r];
\draw[thick] (0,0) arc[start angle=180,end angle=90,radius=\r];
\draw[dashed,thick] (-0.5,0.5) arc[start angle=120,end angle=60,radius=2*\r];

\draw[thick] (-0.5,0.5) -- (-0.5,1.4);
\draw[thick] (0.5,0.5) -- (0.5,1.4);

\draw[->] (-1,0) -- (1,0) node[right]{$\Re$};
\draw[->] (0,-0.1) -- (0,1.41) node[above]{$\Im$};

\node[above] at (-0.5,0.2) {$-\frac12+\frac{i}{2}$};
\node[above] at (0.5,0.2) {${\frac12}+\frac{i}{2}$};

\end{tikzpicture}
\caption{}
\end{subfigure}

\caption{Fundamental domains for (A) $\G_0(2)^+$, (B) $\G_0(3)^+$, (C) $\G_0(1)$ and (D) $\G_0(2)$ respectively.}
\label{fig:fundamentaldomains}
\end{figure}

\medskip
\subsection{The modularity} \label{append:all}
Similar to \zcref{tab:main-a},  \zcref{tab:t3t4list} and \zcref{tab:t6list} collects all rational CM-values of the hauptmodul $t_d$ for $d=3$, $4$, and $6$,  and the modularity of the corresponding hypergeometric Galois representations. 
These $t_d$-values are determined  by the relations between the $j$-invariant and $t_d$ described below, 
together with the fact that  $\Q(j(\tau_d)) \subset \R$ has degree at most two over $\Q$; the corresponding CM points and $j$-values are listed in next subsection. 
	$$
		\begin{diagram}
			\node[1]{\G_0^+(3): t_3}    \arrow{se,l,-}{2} 
			\node[2]{ \SL_2(\Z): \frac{1728}j}   \arrow{sw,l,-}{4} \arrow{se,l,-}{3} 
			\node[2]{\G_0^+(2): t_4}   \arrow{sw,l,-}{2}  \\ 
			\node[2]  {\G_0(3): z} 
			\node[2]   {\G_0(2): t_2} 
		\end{diagram}
		$$
 $$
    t_3= 4z(1-z), \quad    \frac{1728}j=\frac{64z(1-z)^3}{(1+8z)^3}=-\frac{27t_2}{(1-4t_2)^3}, \quad  t_4=\frac{-4t_2}{(t_2-1)^2}.  
    $$

\setlength{\LTcapwidth}{\textwidth}

\begin{table}[ht]
\centering

\renewcommand{\arraystretch}{1.5}
\scalebox{1}{
\begin{tabular}{@{\hspace{1.8em}}c@{\hspace{3.2em}}c}

\begin{subtable}[t]{0.32\textwidth}
\centering
\caption{$d=3$}
\begin{tabular}{|c|c|c|c|}
\hline
$t$ & $\tau_d$ & $f_{d,t}$ & $C_{d,t}$ \\ \hline
$1$ & $\tfrac{i}{\sqrt{3}}$ & $f_{12.3.c.a}$ & $18$ \\
$\tfrac{4}{125}$ & $\tfrac{\sqrt{15}i}{3}$ & $f_{15.3.d.b}$ & $\tfrac{45}{4}$ \\
$\tfrac{1}{2}$ & $\tfrac{i\sqrt{6}}{3}$ & $f_{24.3.h.b}$ & $9$ \\
$-\tfrac{9}{16}$ & $\tfrac{1+\sqrt{3}i}{2}$ & $f_{27.3.b.a}$ & $9$ \\
$-\tfrac{1}{16}$ & $\tfrac{3+\sqrt{51}i}{6}$ & $f_{51.3.c.a}$ & $9$ \\
$-\tfrac{1}{80}$ & $\tfrac{3+\sqrt{75}i}{6}$ & $f_{75.3.c.a}$ & $9$ \\
$-\tfrac{1}{1024}$ & $\tfrac{3+\sqrt{123}i}{6}$ & $f_{123.3.b.a}$ & $9$ \\
$-\tfrac{1}{3024}$ & $\tfrac{3+\sqrt{147}i}{6}$ & $f_{147.3.b.a}$ & $9$ \\
$-\tfrac{1}{250000}$ & $\tfrac{3+\sqrt{267}i}{6}$ & $f_{267.3.b.a}$ & $9$ \\ \hline
\end{tabular}
\end{subtable}

&

\begin{subtable}[t]{0.32\textwidth}
\centering
\caption{$d=4$}
\begin{tabular}{|c|c|c|c|}
\hline
$t$ & $\tau_d$ & $f_{d,t}$ & $C_{d,t}$ \\ \hline
$-\tfrac{256}{3^47^2}$ & $\tfrac{1+\sqrt{7}i}{2}$ & $f_{7.3.b.a}$ & $21$ \\
$1$ & $\tfrac{i}{\sqrt{2}}$ & $f_{8.3.d.a}$ & $24$ \\
$\tfrac{1}{7^4}$ & $\tfrac{3i}{\sqrt{2}}$ & $f_{8.3.d.a}$ & $\tfrac{56}{3}$ \\
$-\tfrac{16}{9}$ & $\tfrac{1+\sqrt{3}i}{2}$ & $f_{12.3.c.a}$ & $12$ \\
$\tfrac{32}{81}$ & $i$ & $f_{16.3.c.a}$ & $12$ \\
$\tfrac{1}{9}$ & $\tfrac{i\sqrt{6}}{2}$ & $f_{24.3.h.a}$ & $12$ \\
$-\tfrac{1}{324}$ & $\tfrac{1+\sqrt{13}i}{2}$ & $f_{52.3.b.a}$ & $12$ \\
$\tfrac{1}{9801}$ & $\tfrac{\sqrt{22}i}{2}$ & $f_{88.3.b.a}$ & $12$ \\
$-\tfrac{1}{25920}$ & $\tfrac{1+5i}{2}$ & $f_{100.3.b.a}$ & $12$ \\
$-\tfrac{1}{777924}$ & $\tfrac{1+\sqrt{37}i}{2}$ & $f_{148.3.b.a}$ & $12$ \\
$\tfrac{1}{96059601}$ & $\tfrac{\sqrt{58}i}{2}$ & $f_{232.3.b.a}$ & $12$ \\ \hline
\end{tabular}
\end{subtable}
\end{tabular}
}
\caption{\centering The modularity for (A) $d=3$ and (B) $d=4$}
\label{tab:t3t4list}
\end{table}

\begin{table}[ht]
\centering
\renewcommand{\arraystretch}{1.4}

\begin{minipage}[t]{0.48\textwidth}
\centering
\caption*{ }

\begin{tabular}{|c|c|c|c|}
\hline
$t$ & $\tau_d$ & $f_{d,t}$ & $C_{d,t}$ \\ \hline
$1$ & $i$ & $f_{144.3.g.a}$ & $9$ \\
$\tfrac{4}{125}$ & $\sqrt{3}i$ & $f_{300.3.g.b}$ & $\tfrac{25}{4}$ \\
$\tfrac{27}{125}$ & $\sqrt{2}i$ & $f_{800.3.g.a}$ & $\tfrac{25}{2}$ \\
$-\tfrac{27}{512}$ & $\tfrac{1+\sqrt{11}i}{2}$ & $f_{704.3.h.a}$ & $16$ \\ \hline
\end{tabular}

\end{minipage}
\hfill
\begin{minipage}[t]{0.48\textwidth}
\centering
\caption*{ }

\begin{tabular}{|c|c|c|}
\hline
$t$ & $\tau_d$ & $f_{d,t}=f_D\otimes\chi_{1-t}$ \\ \hline
$-\tfrac{64}{125}$ & $\frac{1+\sqrt{7}i}{2}$ & $f_{7.3.c.a}\otimes\chi_{1-t}$\\ 
$\tfrac{64}{5^3 17^3}$ & $\sqrt{7}i$ & $f_{7.3.c.a}\otimes\chi_{1-t}$\\ 
$\frac{8}{11^3}$ & $2i$ & $f_{16.3.c.a}\otimes\chi_{1-t}$ \\ 
$-\frac{1}{512}$ & $\frac{1+\sqrt{19}i}{2}$ & $f_{19.3.b.a}\otimes\chi_{1-t}$ \\ 
$-\frac{9}{2^9 5^3}$ & $\frac{1+\sqrt{27}i}{2}$ & $f_{27.3.b.a}\otimes\chi_{1-t}$ \\ 
$-\frac{1}{2^{12}5^3}$ & $\frac{1+\sqrt{43}i}{2}$ & $f_{43.3.b.a}\otimes\chi_{1-t}$ \\ 
$-\frac{1}{2^{9}5^3 11^3}$ & $\frac{1+\sqrt{67}i}{2}$ & $f_{67.3.b.a}\otimes\chi_{1-t}$ \\ 
$-\frac{1}{2^{12}5^3 23^3 29^3}$ & $\frac{1+\sqrt{163}i}{2}$ & $f_{163.3.b.a}\otimes\chi_{1-t}$ \\ \hline
\end{tabular}

\end{minipage}
\caption{\centering The modularity for $d=6$ (see \zcref{cor:d6})}
\label{tab:t6list}
\end{table}

\newpage
\subsection{\texorpdfstring{$j$}{j}-values for certain CM points }

Let $\mathcal{O}$ be an imaginary quadratic  order of discriminant $D$, and $\tau \in \H$ be a point such that the elliptic curve $\C/ (\Z+\tau\Z)$ has CM by  $\mathcal{O}$. In \zcref{tab:jvalues}, we list discriminant of all the  imaginary quadratic orders with class number at most $2$  along with the corresponding  points $\tau$ and their $j$-values.  
\renewcommand{\arraystretch}{1.4}

\begin{longtable}{|c||c|c|}

\hline
disc. of CM order & $\tau$ & $j(\tau)$ \\
\hline
\endfirsthead

\hline
disc. of CM  order & $\tau$ & $j(\tau)$ \\
\hline
\endhead

\hline
\endfoot

\endlastfoot

$-3$  & $\frac{1+\sqrt{3}i}{2}$ & $0$ \\
\hline
$-4$  & $i$ & $12^3$ \\
\hline
$-7$  &$\frac{1+\sqrt{7}i}{2}$  & $-15^3$  \\
\hline
$-8$  & $\sqrt{2}i$ & $20^3$  \\
\hline
$-11$ & $\frac{1+\sqrt{11}i}{2}$ & $-32^3$  \\
\hline
$-12$ & $\sqrt{3}i$ & $54000$  \\
\hline
$-15$ & $\frac{1+\sqrt{15}i}{2}$  & $-\frac{135}{2}(1415 + 637 \sqrt{5})$  \\
& $\frac{1+\sqrt{15}i}{4}$  & $-\frac{135}{2}(1415 - 637 \sqrt{5})$  \\
\hline

$-16$ & $2i$ & $287496$  \\
\hline
$-19$ & $\frac{1+\sqrt{19}i}{2}$  &  $-96^3$ \\
\hline
$-20$ & $\sqrt{5}i$ & $320 (1975 + 884 \sqrt{5})$  \\
& $\frac{1+\sqrt{5}i}{2}$ & $320 (1975 - 884 \sqrt{5})$\\
\hline
$-24$ & $\sqrt{6}i$ &  $1728 (1399 + 988 \sqrt{2})$ \\ 
&$\frac{\sqrt{6}i}{2}$ & $1728 (1399 - 988 \sqrt{2})$\\
\hline
$-27$ & $\frac{1+\sqrt{27}i}{2}$ & $-12288000$  \\
\hline
$-28$ & $\sqrt{7}i$ & $16581375$  \\
\hline
$-32$ & $2\sqrt{2}i$ &  $1000 (26125 + 18473\sqrt{2})$ \\
&$\frac{1+\sqrt{2}i}{2}$ &  $1000 (26125 - 18473\sqrt{2})$ \\
\hline
$-35$ & $\frac{1+\sqrt{35}i}{2}$  & $-163840 (360 + 161 \sqrt{5})$  \\
& $ \frac{1+\sqrt{35}i}{6}$  &  $-163840 (360 - 161 \sqrt{5})$  \\
\hline
$-36$ & $3i$ & $76771008+44330496\sqrt{3}$  \\
& 
$\frac{1+i}3$& $76771008-44330496\sqrt{3}$  \\
\hline
$-40$ & $\sqrt{10}i$  & $8640 (24635 + 11016 \sqrt{5})$  \\
& $\frac{\sqrt{10}i}{2}$  & $8640 (24635 - 11016 \sqrt{5})$  \\
\hline
$-43$ & $\frac{1+\sqrt{43}i}{2}$ & $-960^3$  \\
\hline
$-48$ & $2\sqrt{3}i$ &  $40500(35010+20213\sqrt{3})$\\
 & $\frac{\sqrt{3}i}{2}$ &  $40500(35010-20213\sqrt{3})$\\
\hline
$-51$ & $\frac{1+\sqrt{51}i}{2}$  & $-442368 (6263 + 1519 \sqrt{17})$  \\
&  $\frac{3+\sqrt{51}i}{6}$ & $-442368 (6263 - 1519 \sqrt{17})$  \\
\hline
$-52$ & $\sqrt{13}i$ &  $216000 (15965 + 4428 \sqrt{13})$ \\ 
&$ \frac{1+\sqrt{13}i}{2}$ & $216000 (15965 - 4428 \sqrt{13})$ \\
\hline
$-60$ & $\sqrt{15}i$ & $\frac{135}{2}(274207975 + 122629507 \sqrt{5})$ \\
& $\frac{\sqrt{15}i}{3}$ & $\frac{135}{2}(274207975 - 122629507 \sqrt{5})$ \\
\hline
$-64$ &$4i$  & $54 (761354780 + 538359129\sqrt{2})$  \\
 & $\frac{1+2i}{2}$ & $54 (761354780 - 538359129\sqrt{2})$  \\
\hline
$-67$ & $\frac{1+\sqrt{67}i}{2}$  & $-5280^3$  \\
\hline
$-72$ & $3\sqrt{2}i$ & $8000(23604673+9636536\sqrt 6)$\\
& $\frac{\sqrt{2}i}{3}$ & $8000(23604673-9636536\sqrt 6)$\\
\hline
$-75$ & $\frac{1+5\sqrt{3}i}{2}$& $-884736 (369830 + 165393\sqrt{5})$\\
& $\frac{3+\sqrt{3}i}{10}$& $-884736 (369830 - 165393\sqrt{5})$\\
\hline
$-88$ & $\sqrt{22}i$ &  $216000 (14571395 + 10303524 \sqrt{2})$ \\
& $\frac{\sqrt{22}i}{2}$ &$216000 (14571395 - 10303524 \sqrt{2})$ \\
\hline
$-91$ & $\frac{1+\sqrt{91}i}{2}$ &  $-884736 (5854330 + 1623699\sqrt{13})$ \\ 
& $\frac{3+\sqrt{91}i}{10}$ & $-884736 (5854330 - 1623699\sqrt{13})$  \\
\hline
$-99$ & $\frac{1+3\sqrt{11}i}{2}$ & $-180224 (104359189 + 18166603 \sqrt{33})$  \\
&  $\frac{-1+3\sqrt{11}i}{10}$ & $-180224 (104359189 - 18166603 \sqrt{33})$  \\
\hline
$-100$ & $5i$ & $1728 (12740595841 + 5697769392 \sqrt{5})$  \\
 & $\frac{1+5i}{2}$ & $1728 (12740595841 - 5697769392 \sqrt{5})$  \\
\hline
$-112$ & $2\sqrt{7}i$ & $3375 (40728492440 + 15393923181 \sqrt{7})$  \\
& $\frac{\sqrt{7}i}{2}$ & $3375 (40728492440 - 15393923181 \sqrt{7})$  \\
\hline
$-115$ & $\frac{1+\sqrt{115}i}{2}$ & $-4423680 (48360710 + 21627567 \sqrt{5})$  \\
&  $\frac{5+\sqrt{115}i}{10}$ & $-4423680 (48360710 - 21627567 \sqrt{5})$  \\
\hline
$-123$ & $\frac{1+\sqrt{123}i}{2}$  & $-110592000 (6122264 + 956137\sqrt{41})$  \\
& $\frac{3+\sqrt{123}i}{6}$  & $-110592000 (6122264 - 956137\sqrt{41})$  \\
\hline

$-147$ & $\frac{1+\sqrt{147}i}{2}$ & $-331776000 (52518123 + 11460394 \sqrt{21})$  \\
&  $\frac{3+\sqrt{147}i}{6}$ & $-331776000 (52518123 - 11460394 \sqrt{21})$  \\
\hline
$-148$ & $\sqrt{37}i$ & $216000 (91805981021 + 15092810460 \sqrt{37})$  \\
 &  $\frac{1+\sqrt{37}i}{2}$ & $216000 (91805981021 - 15092810460 \sqrt{37})$  \\
\hline
$-163$  & $\frac{1+\sqrt{163}i}{2}$ & $-640320^3$  \\
\hline
$-187$ & $\frac{1+\sqrt{187}i}{2}$ & $-940032000 (2417649815 + 586366209 \sqrt{17})$ \\
 & $\frac{-3+\sqrt{187}i}{14}$  & $-940032000 (2417649815 - 586366209 \sqrt{17})$ \\
\hline
$-232$ & $\sqrt{58}i$ & $216000 (1399837865393267 + 259943365786104\sqrt{29})$ \\
& $\frac{\sqrt{58}i}{2}$ & $216000 (1399837865393267 - 259943365786104\sqrt{29})$ \\
\hline
$-235$ & $\frac{1+\sqrt{235}i}{2}$ & $-5887918080 (69903946375 + 31261995198\sqrt{5})$ \\
& $\frac{5+\sqrt{235}i}{10}$  & $-5887918080 (69903946375 - 31261995198\sqrt{5})$ \\
\hline
$-267$ & $\frac{1+\sqrt{267}i}{2}$  & $-55296000 (177979346192125 +18865772964857 \sqrt{89})$  \\
 &  $\frac{3+\sqrt{267}i}{6}$  & $-55296000 (177979346192125 - 18865772964857 \sqrt{89})$  \\
\hline
$-403$ & $\frac{1+\sqrt{403}i}{2}$ & $-110592000 (11089461214325319155\ + $\\
&&$3075663155809161078 \sqrt{13})$  \\
& $\frac{-9+\sqrt{403}i}{22}$ & $-110592000 (11089461214325319155\ - $\\
&&$3075663155809161078 \sqrt{13})$  \\

\hline
$-427$ & $\frac{1+\sqrt{427}i}{2}$  & $-147197952000 (53028779614147702 \ +$\\
&&$ 6789639488444631 \sqrt{61})$  \\
&  $\frac{7+\sqrt{427}i}{14}$ & $-147197952000 (53028779614147702 \ -$\\
&&$ 6789639488444631 \sqrt{61})$  \\
\hline
\caption{$j$-values for imaginary quadratic orders with class number at most $2$}
\label{tab:jvalues}
\end{longtable}


\bibliographystyle{abbrv}
\bibliography{Ref}

 \address{Department of Mathematics, Louisiana State University, Baton Rouge, LA 70803, USA.}

 \email{\href{mailto:parora2@lsu.edu}{parora2@lsu.edu}}, {\url{https://sites.google.com/view/paresh-singh-arora/home}}

 \email{\href{mailto:kmonda1@lsu.edu}{kmonda1@lsu.edu}}, {\url{https://sites.google.com/view/kmondal}}

 \email{\href{mailto:ftu@lsu.edu}{ftu@lsu.edu}}, {\url{https://sites.google.com/view/ft-tu/}}

 \address{Institute of Science and Engineering, Kanazawa University, Kakuma, Kanazawa, 920-1192, JAPAN. }

 \email{\href{mailto:akio.nakagawa.math@icloud.com}{akio.nakagawa.math@icloud.com}} {\url{}}

\end{document}